\documentclass[11pt, reqno]{amsart}
\usepackage[colorlinks=true, pdfstartview=FitV, linkcolor=blue, 
citecolor=blue]{hyperref}

\usepackage[letterpaper, margin=1.3in]{geometry}

\usepackage{amssymb,amsmath,amsfonts}
\usepackage{bm}
\usepackage{enumitem}
\usepackage{array}

\usepackage{multicol}
\usepackage{graphicx}

\allowdisplaybreaks

\newcommand{\N}{\mathbb{N}}
\newcommand{\R}{\mathbb{R}}
\newcommand{\C}{\mathbb{C}}
\newcommand{\Z}{\mathbb{Z}}
\newcommand{\Q}{\mathbb{Q}}

\newcommand{\rar}{\rightarrow}
\newcommand{\mc}{\mathcal}

\newcommand{\bp}{\bm{p}}
\newcommand{\balph}{\bm{\alpha}}

\newtheorem*{theorem*}{Theorem}
\newtheorem{theorem}{Theorem}
\newtheorem{lemma}[theorem]{Lemma}

\newcommand{\rvline}{\hspace*{-\arraycolsep}\vline\hspace*{-\arraycolsep}}

\newcommand{\cW}{\mc{W}}

\newcommand{\NN}{\mathrm{N}}
\newcommand{\Tr}{\mathrm{Tr}}
\newcommand{\diam}{\mathrm{diam}}

\newcommand{\NA}{\mathrm{NA}_\eta(\balph)}

\begin{document}
	
	\title[Accumulation points of normalized approximations]{Accumulation points of normalized approximations}
	\author{Kavita Dhanda, Alan Haynes}

	\thanks{Research supported by NSF grant DMS 2001248.\\
		\phantom{A..}MSC 2020: 11J68, 11J13, 11K60}
	

	\begin{abstract}
		Building on classical aspects of the theory of Diophantine approximation, we consider the collection of all accumulation points of normalized integer vector translates of points $q\balph$ with $\balph\in\R^d$ and $q\in\Z$. In the first part of the paper we derive measure theoretic and Hausdorff dimension results about the set of $\balph$ whose accumulation points are all of $\R^d$. In the second part we focus primarily on the case when the coordinates of $\balph$ together with $1$ form a basis for an algebraic number field $K$. Here we show that, under the correct normalization, the set of accumulation points displays an ordered geometric structure which reflects algebraic properties of the underlying number field. For example, when $d=2$, this collection of accumulation points can be described as a countable union of dilates (by norms of elements of an order in $K$) of a single ellipse, or of a pair of hyperbolas, depending on whether or not $K$ has a non-trivial embedding into $\C$.
	\end{abstract}
	
	\maketitle

	\section{Introduction}\label{sec.Intro}
	Suppose that $d\in\N$ and that $\bm{\alpha}\in\R^d$. A large portion of classical Diophantine approximation is concerned with understanding small values of the quantities
	\begin{equation*}
		q\bm{\alpha}-\bm{p}\in\R^d,
	\end{equation*}
	where $q\in\Z$ and $\bm{p}\in\Z^d$. A first version of this problem, which is sufficient for many applications, is to accurately determine, for $\eta\in\R$, the `sizes' of the sets $\cW(\eta)$ of $\eta$-\textit{well approximable points}, and their complements $\cW(\eta)^c$. For $\eta\in\R$, $\mc{W}(\eta)$ is defined to be the collection of $\balph\in\R^d$ for which $\bm{0}$ is an accumulation point of the set
	\[\mathrm{NA}_\eta(\balph)=\{|q|^{\eta}\left(q\bm{\alpha}-\bm{p}\right):q\in\Z\setminus\{0\},\bp\in\Z^d\}.\]
	We will call the elements of $\mathrm{NA}_\eta(\balph)$ the $\eta$-\textit{normalized approximations} to $\balph$.
	
	For each $\eta$, the set $\mc{W}(\eta)$ satisfies a zero-full law, so that $\lambda(\cW(\eta))=0$ or $\lambda(\cW(\eta)^c)=0$, where $\lambda$ denotes Lebesgue measure on $\R^d$. In cases where the Lebesgue measure is zero, we seek finer information about Hausdorff dimensions of the corresponding sets. This is all well understood, primarily due to foundational work of Khintchine, Groshev, Jarn\'{i}k, Besicovitch, Schmidt, and Cassels, which we discuss below.
	
	In line with the classical theory described above, it is natural and desirable to understand the collection of all accumulation points $\mc{A}_\eta(\balph)$ of $\NA$, both for generic choices of $\balph$ and for specific choices with particular arithmetical properties. For example, the density of $\Q^d$ in $\R^d$ guarantees that $\mc{A}_{-1}(\balph)=\R^d$, for all $\balph\in\R^d$. On the other hand, Kronecker's theorem, a much less trivial result, tells us that $\mc{A}_0(\balph)=\R^d$ if and only if the numbers $1,\alpha_1,\ldots ,\alpha_d$ are $\Q$-linearly independent. However, perhaps surprisingly, apart from these two fundamental examples, there do not appear to be any non-trivial cases of $\eta$ (i.e. with $\eta>0$) for which this problem is well understood.
	
	In this paper we will study the sets $\mc{A}_\eta(\balph)$ from two main points of view. First, for $\eta\in\R$, we will prove measure theoretic and Hausdorff dimension results about the sets $\mc{D}(\eta)$ consisting of all $\balph\in\R^d$ for which $\mc{A}_\eta(\balph)=\R^d$, as well as results about the complementary sets $\mc{D}(\eta)^c$. As we will see, in analogy with the classical case, for $\eta\le 1$ the property $\mc{A}_\eta(\balph)=\R^d$ is satisfied for Lebesgue almost every $\balph\in\R^d$. However, for $d\ge 2$ and $0<\eta\le 1$, the sets $\mc{D}(\eta)^c$ are related to what are called sets of singular points in Diophantine approximation, and are much more mysterious and interesting than the corresponding sets $\mc{W}(\eta)^c$ from the classical case.
	
	Our second point of view will study particular choices of $\eta$ and $\balph$ for which the sets $\mc{A}_\eta(\balph)$ exhibit well ordered geometric structure. We will focus primarily on the cases when $d=1$ or $2$, $\eta=1/d$, and the coordinates of $\balph$ together with 1 form a $\Q$-basis for an algebraic number field. Using tools from algebraic number theory, we will show that if $d=1$ then $\mc{A}_{1}(\balph)$ is a discrete set determined by norms of elements in an order in the corresponding quadratic field. If $d=2$ then $\mc{A}_{1/2}(\balph)$ is a union of countably many dilations (determined by norms of elements in an order in the corresponding cubic field) of either a single ellipse or a pair of hyperbolas, depending on whether or not the corresponding cubic field has a nontrivial complex embedding.
	
	We will now explain the results just mentioned in more detail. For the reader's benefit, a self-contained summary of our notation is provided at the beginning of Section \ref{sec.Notation}.
	
	\subsection{Classical results} For comparison with our main theorems, we first present known results about the sets $\mc{W}(\eta)$. These results are summarized in Table \ref{fig.Classical} below.
	
	For $\bm{x}\in\R^d$ write
	\[|\bm{x}|=\max\left\{|x_1|,\ldots ,|x_d|\right\}\quad\text{and }\quad \|\bm{x}\|=\min_{\bm{p}\in\Z^d}\left|\bm{x}-\bm{p}\right|.\]
	It may be helpful in what follows to observe that a point $\bm{\gamma}\in\R^d$ is an accumulation point of $\NA$ if and only if, for every $\epsilon>0$, there exist $q\in\Z\setminus\{0\}$ and $\bm{p}\in\Z^d$ for which
	\begin{equation}\label{eqn.AccPtForm1}
		0<\left|q\bm{\alpha}-\frac{\bm{\gamma}}{|q|^{\eta}}-\bm{p}\right|<\frac{\epsilon}{|q|^{\eta}}.
	\end{equation}
	If $\eta<0$ then, for any $\bm{\gamma}\in\R^d$ and $\epsilon>0$, the right hand side of this inequality will be greater than $1$ for all sufficiently large $q$. Therefore it follows trivially that $\mc{W}(\eta)=\mc{D}(\eta)=\R^d$ for $\eta<0$. If $\eta\ge 0$ then, for $\epsilon<1/2$, the existence of $q$ and $\bm{p}$ satisfying \eqref{eqn.AccPtForm1} is equivalent to the existence of an integer $q\in\Z\setminus\{0\}$ for which
	\begin{equation}\label{eqn.AccPtForm2}
		0<\left\|q\bm{\alpha}-\frac{\bm{\gamma}}{|q|^{\eta}}\right\|<\frac{\epsilon}{|q|^{\eta}}.
	\end{equation}
	
	The multidimensional version of Dirichlet's theorem \cite[Chapter 1, Theorem VI]{Cass1957} tells us that, for any $\bm{\alpha}\in\R^d$, there are infinitely many $q\in\N$ for which
	\begin{equation*}
		q^{1/d}\|q\balph\|< 1.
	\end{equation*}
	When $\balph\in\Q^d$ the set $\{q\balph+\bm{p} : q\in\Z,\bm{p}\in\Z^d\}$ is discrete (and $\|q\balph\|=0$ infinitely often). Taking this into consideration, Dirichlet's theorem implies that, for all $0\le \eta<1/d$,
	\[\mc{W}(\eta)=\R^d\setminus\Q^d.\]
	
	Combining Dirichlet's theorem with a higher dimensional version of Cassels's lemma (see Lemma \ref{lem.CassHighD}), we have that for almost every $\bm{\alpha}\in\R^d$ with respect to Lebesgue measure,
	\begin{equation*}
		\inf_{q\in\N} q^{1/d}\|q\balph\|=0.
	\end{equation*}
	This last assertion (in fact a stronger statement) also follows from the Khintchine-Groshev theorem (see Theorem \ref{thm.KhinGros}), and it implies that $\lambda(\mc{W}(1/d))=\infty$ and that $\lambda(\mc{W}(1/d)^c)=0$.
	
	\vspace*{0bp}
	\renewcommand{\arraystretch}{1.8}
	\begin{table}[h]
		\caption{Summary of results about $\mc{W}(\eta)$ and $\mc{W}(\eta)^c$, for $\eta\in\R$ and $d\in\N$.}\label{fig.Classical}\vspace*{-10bp}
		\[\begin{array}{c|c|c}
			\rule[-8bp]{0pt}{0bp}\eta<0& \mc{W}(\eta)=\R^d & \mc{W}(\eta)^c=\emptyset \\
			\hline \rule{0pt}{21bp}\rule[-14bp]{0pt}{0bp} 0\le\eta<1/d & \begin{split}
				&\mc{W}(\eta)=\R^d\setminus\Q^d\\&\hspace*{14bp}\text{(Dirichlet)}
			\end{split} & \mc{W}(\eta)^c=\Q^d \\ 
			\hline \rule{0pt}{22bp}\rule[-14bp]{0pt}{0bp} \eta=1/d & \begin{split}
				&\lambda(\mc{W}(\eta))=\infty, ~\lambda(\mc{W}(\eta)^c)=0\\&\hspace*{20bp}\text{(Khintchine, Groshev)}
			\end{split} & \begin{split}	&\dim_H(\mc{W}(\eta)^c)=d\\&\hspace*{3bp}\text{(Jarn\'{i}k, Schmidt)} \end{split} \\
			\hline \rule{0pt}{28bp}\rule[-16bp]{0pt}{0bp}
			\eta>1/d& \begin{split}	&\dim_H(\mc{W}(\eta))=\frac{d+1}{1+\eta}\\&\hspace*{3bp}\text{(Jarn\'{i}k, Besicovitch)} \end{split} & \begin{split}
				&\lambda(\mc{W}(\eta)^c)=\infty, \phantom{\frac{d}{\eta}}\hspace*{-6bp}\lambda(\mc{W}(\eta))=0\\&\hspace*{33bp}\text{(Borel-Cantelli)}
			\end{split}
		\end{array}\]	
	\end{table}

	To better understand the size of $\mc{W}(1/d)^c$, note that a point $\bm{\alpha}\in\R^d\setminus\Q^d$ will be in $\mc{W}(1/d)^c$ if and only if
	\begin{equation}\label{eqn.BadDef}
		\inf_{q\in\N} q^{1/d}\|q\balph\|>0.	
	\end{equation}
	Such points are called \textit{badly approximable}. It was proved by Jarn\'{i}k \cite{Jarn1928} for $d=1$ and by Schmidt \cite{Schm1969} for $d\ge 2$ that the set of all badly approximable points in $\R^d$ has Hausdorff dimension $d$. 
	
	For $\eta>1/d$ a straightforward application of the Borel-Cantelli lemma implies that $\lambda(\mc{W}(\eta))=0$. Furthermore, theorems of Jarn\'{i}k \cite{Jarn1929} and Besicovitch \cite{Besi2006} for $d=1$, and of Jarn\'{i}k \cite{Jarn1931} for $d\ge 2$ imply that
	\begin{equation*}
		\dim_H(\mc{W}(\eta))=\frac{d+1}{1+\eta}\quad\text{for}\quad \eta>1/d.
	\end{equation*}

	\subsection{Measure theory and dimension results for $\mc{D}(\eta)$ and $\mc{D}(\eta)^c$} Now we present results about the sets $\mc{D}(\eta)$ and $\mc{D}(\eta)^c$. To our knowledge, most of the non-trivial results in this subsection have not appeared in previous work. Exceptions are the cases of $\eta=-1$ and $0$ (mentioned above), and the case of $d\ge 3$ in Theorem \ref{thm.LebMeasDens} (see the comments after the statement of the theorem). The results presented here are summarized in Table \ref{fig.NewResults} below.
	
	It was already pointed out after equation \eqref{eqn.AccPtForm1} that $\mc{D}(\eta)=\R^d$ for $\eta<0.$ Note also that
	\begin{equation}\label{eqn.DWInclusions}
		\mc{D}(\eta)\subseteq\mc{W}(\eta)\quad\text{for any}\quad \eta\in\R.
	\end{equation}
	This gives that 
	\begin{align*}
		\lambda(\mc{D}(\eta))=0\quad\text{for}\quad \eta>1/d.
	\end{align*}
	The Lebesgue measure theory of the sets $\mc{D}(\eta)$ is therefore completed by the following theorem.
	\begin{theorem}\label{thm.LebMeasDens}
		For $d\in\N$ and $\eta\le 1/d$ we have that
		\[\lambda(\mc{D}(\eta))=\infty\quad\text{and}\quad\lambda(\mc{D}(\eta)^c)=0.\]
	\end{theorem}
	Although this appears similar to the corresponding result for $\mc{W}(\eta)$ (the Khintchine-Groshev theorem), it does not follow as an immediate corollary. It seems that the result of Theorem \ref{thm.LebMeasDens} is natural enough that it could have been observed in prior literature, however, we were unable to locate a reference. For $d\ge 3$ it can be derived from more general recent work of Allen and Ramirez (see \cite[Theorem 2]{AlleRami2022}). In any case, we will give a self contained and relatively simple proof of Theorem \ref{thm.LebMeasDens}, for all $d\in\N$, in Section \ref{sec.ThmDens}.
	
	To discuss our results for $0\le \eta<1/d$, let $K_d$ denote the set of $\balph\in\R^d$ for which the numbers $1,\alpha_1,\ldots, \alpha_d$ are $\Q$-linearly independent (this is the notation used in \cite{Bake1992}; the $K$ here can be thought of as standing for `Kronecker'). By Kronecker's theorem in Diophantine approximation \cite[Chapter 3, Theorem IV]{Cass1957}, we have that $\mc{D}(0)=K_d.$ Furthermore, for $\balph\notin K_d$, the collection of points $\{q\balph:q\in\Z\}$ lies on an affine hyperplane of $\R^d$ which is defined over $\Q$, so that modulo $\Z^d$ this collection lies on a proper rational sub-torus of $\R^d/\Z^d.$ It follows from this and the characterization given in \eqref{eqn.AccPtForm2} that $K_d^c\subseteq \mc{D}(\eta)^c$ for all $\eta\ge 0$.
	
	From the discussion in the previous paragraph we see that for $\eta>0$ the set $K_d^c$ plays a role for $\mc{D}(\eta)^c$ which is analogous to that played by the set $\Q^d$ for $\mc{W}(\eta)^c$. However, as indicated by our next result, for $d\ge 2$ and $0<\eta<1/d$, the sets $\mc{D}(\eta)^c$ turn out to contain many other points as well.
	\begin{theorem}\label{thm.SingDens}
		Suppose that $d\in\N$ and $0<\eta<1/d$. If $d=1$ then $\mc{D}(\eta)^c=K_1^c=\Q$. If $d\ge 2$ then
		\begin{equation}\label{eqn.HDSingLowBd}
			\dim_H(\mc{D}(\eta)^c\setminus K_d^c)\ge \frac{d}{1+1/\eta}
		\end{equation}
		and
		\begin{equation}\label{eqn.HDSingUpBd}
			\dim_H(\mc{D}(\eta)^c\setminus K_d^c)\le d-2+\frac{2(d+1)((1+\eta)d-1)}{(1+\eta)d^2}.
		\end{equation}
	\end{theorem}
	We will show in the proof of Theorem \ref{thm.SingDens} that, for $0<\eta<1/d$, the elements of $\mc{D}(\eta)^c\setminus K_d^c$ are a subset of the collection of \textit{singular points} in $\R^d$ (these will be defined in Section \ref{sec.Notation}). The study of singular points has a long history, going back at least to Khintchine \cite{Khin1926}, and it is also an area of active research in Diophantine approximation and dynamical systems (see \cite{BereGuanMarnRamiVela2022,BugeCheuChev2019,Cheu2011,LiaoShiSolaTama2020} and references within).

	\vspace*{0bp}
	\renewcommand{\arraystretch}{1.8}
	\begin{table}[h]
		\caption{Summary of results about $\mc{D}(\eta)$ and $\mc{D}(\eta)^c$, for $\eta\in\R$ and $d\in\N$.}\label{fig.NewResults}\vspace*{-10bp}
		\[\begin{array}{c|c|c}
			\rule[-8bp]{0pt}{0bp}\eta<0& \mc{D}(\eta)=\R^d & \mc{D}(\eta)^c=\emptyset \\
			\hline \rule{0pt}{21bp}\rule[-14bp]{0pt}{0bp} \eta=0 & \begin{split}
				&\mc{D}(\eta)=K_d\\&\hspace*{-1.5bp}\text{(Kronecker)}
			\end{split} & \mc{D}(\eta)^c=K_d^c\\
			\hline \rule{0pt}{21bp}\rule[-14bp]{0pt}{0bp} 0<\eta<1/d & \begin{split}
				&\lambda(\mc{D}(\eta))=\infty\\&\hspace*{2.5bp}\text{(Theorem \ref{thm.LebMeasDens})}
			\end{split} & \begin{split}&0<\dim_H(\mc{D}(\eta)^c\setminus K_d^c)<d\\&\hspace*{10bp}\text{($d\ge 2$, cf. Theorem \ref{thm.SingDens})}\end{split} \\ 
			\hline \rule{0pt}{22bp}\rule[-14bp]{0pt}{0bp} \eta=1/d & \begin{split}
				&\lambda(\mc{D}(\eta))=\infty, ~\lambda(\mc{D}(\eta)^c)=0\\&\hspace*{37.8bp}\text{(Theorem \ref{thm.LebMeasDens})}
			\end{split} & \begin{split}	&\dim_H(\mc{D}(\eta)^c)=d\\&\hspace*{3bp}(\mc{W}(\eta)^c\subseteq\mc{D}(\eta)^c)\end{split} \\
			\hline \rule{0pt}{28bp}\rule[-16bp]{0pt}{0bp}
			\eta>1/d& \begin{split}	&\dim_H(\mc{D}(\eta))=\frac{d+1}{1+\eta}\\&\hspace*{21.5bp}\text{(Theorem \ref{thm.JarBesDens})} \end{split} & \begin{split}
				&\lambda(\mc{D}(\eta)^c)=\infty, \phantom{\frac{d}{\eta}}\hspace*{-6bp}\lambda(\mc{D}(\eta))=0\\&\hspace*{32bp}(\mc{D}(\eta)\subseteq\mc{W}(\eta))
			\end{split}
		\end{array}\]	
	\end{table}
	
	It is clear that the bounds in Theorem \ref{thm.SingDens} imply the weaker statement recorded in Table \ref{fig.NewResults} that, for $d\ge 2$ and $0<\eta<1/d$,
	\[0<\dim_H(\mc{D}(\eta)^c\setminus K_d^c)<d.\]
	Note that, as $\eta\rar (1/d)^-$, the right hand side of \eqref{eqn.HDSingUpBd} tends to $d$. This is consistent with the fact that, by \eqref{eqn.DWInclusions},
	\[\mc{W}(1/d)^c\subseteq\mc{D}(1/d)^c,\]
	and therefore
	\[\dim_H(\mc{D}(1/d)^c)=d.\]
	
	Finally, for $\eta>1/d$ we will establish the following analogue of the Jarn\'{i}k-Besicovitch theorem for the sets $\mc{D}(\eta)$.
	\begin{theorem}\label{thm.JarBesDens}
		For $d\in\N$ and $\eta>1/d$,
		\[\dim_H(\mc{D}(\eta))=\frac{d+1}{1+\eta}.\]
	\end{theorem}
	Although there are more detailed questions which can be asked about the measure and dimension theoretic properties of the sets $\mc{D}(\eta)$ and their complements, Theorems \ref{thm.LebMeasDens}-\ref{thm.JarBesDens} at least bring our state of knowledge roughly in line with what is known about the classical sets $\mc{W}(\eta)$.

	\subsection{Algebraic orbits}
	Suppose we choose $d\in\N, \eta>0$, and $\balph\in K_d$ with $\balph\notin \mc{D}(\eta)$. In general, there does not seem to be a reason to expect that the (non-dense) collection of accumulation points $\mc{A}_\eta(\balph)$ should possess any discernible geometric structure. However, in order to actually construct examples of such $\balph$, one with a modest amount of experience is likely to be led first of all to take $\eta=1/d$ and to choose $\balph\in\R^d$ so that $1,\alpha_1,\ldots,\alpha_d$ form a $\Q$-basis for an algebraic number field of degree $d+1$. This is in part because, by classical work of Perron \cite{Perr1921}, such $\balph$ are always badly approximable, and therefore are elements of $K_d\cap \mc{D}(1/d)^c$.

	Figure \ref{fig.EllsHyps} shows the surprising results of plotting $(1/2)$-normalized approximations to pairs of cubic numbers of the type just described. We remark that it is a nontrivial task to generate interesting plots of normalized approximations in these cases. The integers which give rise to normalized approximations lying in a bounded region grow geometrically, so exhaustive searches do not yield much data. Instead, the figures we present here, which show a proper subset of the normalized approximations with $q$ and $|\bm{p}|$ up to some bound, were generated using the results developed in Section \ref{sec.NormAlg} below. Even so, the computations to produce the plots required a significant amount of precision in the calculations. As we will see, the details of these plots reveal features of the underlying number fields. For example, observe that there is a `missing' ellipse in the 7th ring of the plot on the left. This is a manifestation of the fact that $7$ is an inert prime in the ring of integers of $\Q(2^{1/3})$, which implies that there is no element of norm 7 in the ring.
	
	\begin{figure}[h]
		\caption{$(1/2)$-normalized approximations to $\balph=(2^{1/3},2^{2/3})$ (left) and to $\balph=(\alpha,\alpha^2)$, where $\alpha$ is the positive root of $x^3+x^2-2x-1$ (right).\\}\label{fig.EllsHyps}
		\centering
		\includegraphics[width=\textwidth]{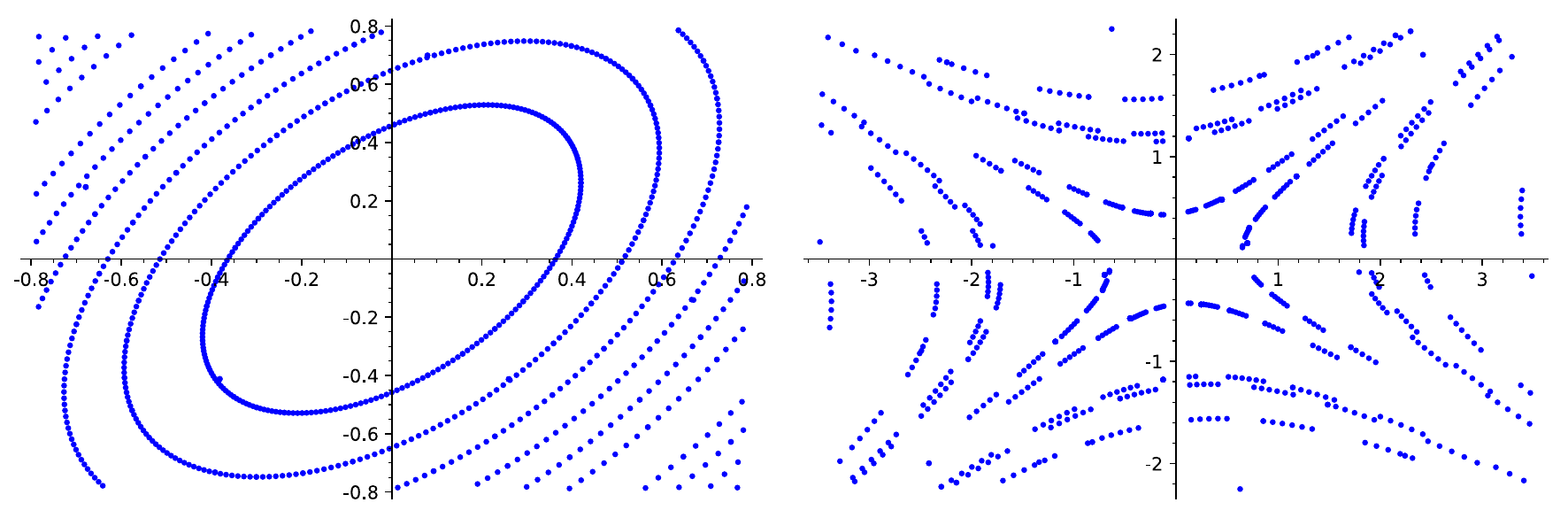}
	\end{figure}

	As a first explanation for these types of plots, results of Adams \cite{Adam1969a,Adam1969b,Adam1971}, Chevallier \cite{Chev2002}, Ito and Yasutomi \cite{ItoYasu2006}, and  Hensley \cite[Chapter 7]{Hens2006} have together demonstrated that if $1,\alpha_1,\ldots,\alpha_d$ form a basis for an algebraic number field then the collection $\mathrm{NA}_{1/d}(\balph)$ of ($1/d$)-normalized approximations to $\balph$ lie near specific curves and surfaces in $\R^d$ (see Theorem 7.3 and Section 7.8 in \cite{Hens2006}). In the second part of this paper we will give a complete description of the sets of accumulation points $\mc{A}_{1/d}(\balph)$ of normalized approximations, in the algebraic setting just described, when $d=1$ and $2$. First of all, we have the following theorem when $d=1$.
	\begin{theorem}\label{thm.QuadratApps}
		Suppose that $\alpha=a+b\sqrt{D}\in\R$, with $a,b\in\Q$, $b\not=0$, and $D\ge 2$ a squarefree integer, and let $K=\Q(\alpha)$. Then 	
		\[\mc{A}_1(\alpha)=\left\{\pm 2b\sqrt{D}\phantom{\cdot}\NN(s) : s\in\Lambda^*\setminus\{0\}\right\}\]
		where $\Lambda^*$ is the lattice in $K$ dual to \[\Lambda=\Z+\alpha\Z,\]
		with respect to the trace form.
	\end{theorem}
	In the statement of the theorem, $\NN(\cdot)$ denotes the $K/\Q$ norm. Definitions of `lattice,' `dual lattice,' and `trace form' are given in the next section. 
	
	It is conceivable that Theorem \ref{thm.QuadratApps} could have been observed in previous literature, since it is reasonable to expect that there are elementary ways to prove this theorem. However, with a view toward generalizing to higher dimensions, we will instead present a proof which is cast in the language of algebraic number theory. With this as a backdrop, we will then prove the following result for $d=2$.
	\begin{theorem}\label{thm.CubicApps}
		Suppose that $\bm{\alpha}\in\R^2$ and that $\{1,\alpha_1,\alpha_2\}$ is a basis for a cubic field $K/\Q$, and let $\Lambda^*$ be the lattice in $K$ dual to \[\Lambda=\Z+\alpha_1\Z+\alpha_2\Z,\] with respect to the trace form.
		\begin{itemize}[itemsep=5bp,parsep=10bp, topsep=0bp]
			\item If $K$ has a nontrivial complex embedding then there is an ellipse $\mc{C}_K\subset\R^2$ with the property that
			\[\mc{A}_{1/2}(\balph)=\bigcup_{s\in \Lambda^*\setminus\{0\}}\left|\NN (s)\right|^{1/2}\mc{C}_K.\]
			\item If $K$ is totally real then there is a hyperbola $\mc{C}_K^+\subset\R^2$ with the property that
			\[\mc{A}_{1/2}(\balph)=\bigcup_{s\in \Lambda^*\setminus\{0\}}\left|\NN (s)\right|^{1/2}\ell_{s}(\mc{C}_K^+),\]
			where for each $s\in \Lambda^*\setminus\{0\}$, the map $\ell_s:\R^2\rar\R^2$ is either the identity map, or a linear map taking $\mc{C}_K^+$ to its conjugate hyperbola.
		\end{itemize}
	\end{theorem}
	We will see in the proof below that, in the totally real case of this theorem, the map $\ell_s$ is determined by whether or not the two real conjugates of $s$ have the same sign. In most cases, for each $s\in\Lambda^*\setminus\{0\}$ there will be another element $\tilde{s}\in\Lambda^*\setminus\{0\}$ with $\NN(s)=\NN(\tilde{s})$ and for which $\ell_s\not=\ell_{\tilde{s}}$. When this happens, the collection of accumulation points $\mc{A}_{1/2}(\balph)$ will be a union of dilates of a single pair of hyperbolas (as in the plot on the right in Figure \ref{fig.EllsHyps}). However, although it seems to be a rarer phenomenon, this can actually fail to happen. As an example, let $\balph=(\alpha,\alpha^2)$, where $\alpha$ is the smallest positive root of $x^3-19x+21$. The field $K=\Q(\alpha)$, which we discovered using the extensive data recorded in \cite{EnnoTuru1985}, has discriminant 15529. It is totally real and has the exotic feature that every positive unit in its ring of integers has only positive conjugates. This implies that if $s$ and $\tilde{s}$ are any elements of $\Lambda^*\setminus\{0\}$ with the same norm, then $\ell_s=\ell_{\tilde{s}}$. In other words, the points of $\mc{A}_{1/2}(\balph)$ are unions of scaled copies of each of the single hyperbolas in a pair of hyperbolas, and only one hyperbola from the pair appears for each possible norm. Although it is not trivial to discern this by casual observation, a careful examination of Figure \ref{fig.Hyps2} (comparing ratios of scaled hyperbolas with ratios of scaled conjugate hyperbolas) will reveal the behavior that we are describing.

	\begin{figure}[h]
	\caption{$(1/2)$-normalized approximations to $\balph=(\alpha,\alpha^2)$, where $\alpha$ is the smallest positive root of $x^3-19x+21$.\\}\label{fig.Hyps2}
	\centering
	\includegraphics[width=.6\textwidth]{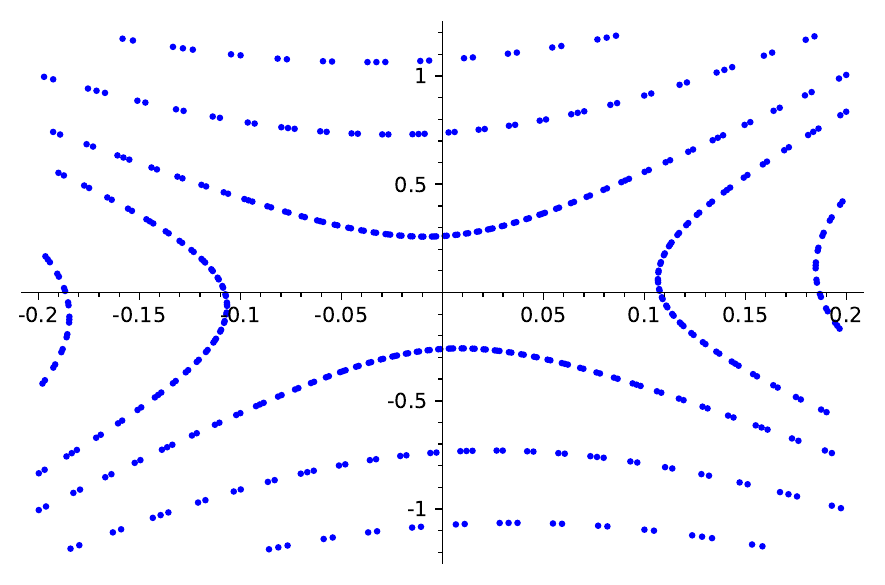}
\end{figure}

	Our proofs of Theorems \ref{thm.QuadratApps} and \ref{thm.CubicApps} will use a combination of ideas from the above mentioned authors who have worked on these types of problems, together with ideas from Cassels and Swinnerton-Dyer \cite{CassSwin1955} and Peck's work \cite{Peck1961} related to the Littlewood conjecture.
	
	The machinery which is going to be developed in this paper (in particular, in Section \ref{sec.NormAlg}) will provide us with a relatively nice characterization of the collection of $(1/d)$-approximations to $\balph\in\R^d$ which lie in some fixed bound, for any $d\in\N$, when the coordinates of $\balph$ together with 1 form a basis for a number field (see Lemmas \ref{lem.NormAppChar1} and \ref{lem.NormAppChar2}). Therefore it may be possible to extend our proofs to establish the obvious generalizations of Theorems \ref{thm.QuadratApps} and \ref{thm.CubicApps} for accumulation points of normalized approximations to algebraic numbers in dimensions $d\ge 3$. However, there appear to be some non-trivial details related to the distribution of units in the corresponding fields that would need to be worked out.
	
	This paper is structured as follows. In Section \ref{sec.Notation} we present notation and background material from Diophantine approximation, measure theory, Hausdorff dimension theory, and algebraic number theory. In Sections \ref{sec.ThmDens}-\ref{sec.JarBesThm} we present our proofs of Theorems \ref{thm.LebMeasDens}-\ref{thm.JarBesDens}. In Section \ref{sec.NormAlg} we establish a general framework for studying $(1/d)$-normalized approximation of algebraic numbers of the type described above. Finally, in Sections \ref{sec.QuadThm} and \ref{sec.CubicThm} we prove Theorems \ref{thm.QuadratApps} and \ref{thm.CubicApps}.
	
	\noindent\textbf{Acknowledgements:} We would like to thank Shigeki Akiyama, Pierre Arnoux, Val\'{e}rie Berth\'{e}, and Lenny Fukshansky for useful advice and comments. We also thank the Mathematisches Forschungsinstitut Oberwolfach for their hospitality during part of our work on this project.
	
	\section{Notation and preliminary results}\label{sec.Notation}
	
	\noindent \textbf{Notation:} We denote vectors in $\R^d$ using bold face, and their components with appropriate subscripts. For convenience in typesetting, column vectors are often denoted as transposes of row vectors. For $\bm{x}\in\R^d$, we set
	\[|\bm{x}|=\max\left\{|x_1|,\ldots ,|x_d|\right\}\quad\text{ and }\quad \|\bm{x}\|=\min_{\bm{p}\in\Z^d}\left|\bm{x}-\bm{p}\right|.\]
	For $\bm{x},\bm{y}\in\R^d$, we write $\bm{x}\bm{y}$ for the standard Euclidean inner product of $\bm{x}$ and $\bm{y}$. When working in $\R^d$, the symbol $\lambda$ denotes $d$-dimensional Lebesgue measure. For $x\in\R$ and $\delta>0$, the symbol $B(x,\delta)$ denotes the open ball of radius $\delta$ centered at $x$. We use $\ll$ and $\gg$ for the standard Vinogradov notation.  For any set $A\subseteq\R^d$, we write $\dim_H(A)$ for the Hausdorff dimension of $A$. If $A\subseteq\R^d$ and $\delta>0$, we say that $A$ is \textit{$\delta$-dense} in $\R^d$ if, for every $\bm{x}\in\R^d$, there exists a point $\bm{y}\in A$ with $|\bm{x}-\bm{y}|\le \delta$. We say that $A\subseteq\R^d$ is \textit{uniformly discrete} if there exists an $r>0$ such that $|\bm{x}-\bm{y}|>r$ for all distinct points $\bm{x},\bm{y}\in A.$
	
	For convenience we also collect together here the following notation from the introduction. For $\balph\in\R^d$ and $\eta\in\R$, the set of $\eta$-\textit{normalized approximations} to $\balph$ is
	\[\mathrm{NA}_\eta(\balph)=\{|q|^{\eta}\left(q\bm{\alpha}-\bm{p}\right):q\in\Z\setminus\{0\},\bp\in\Z^d\}.\]
	The collection of all accumulation points of $\mathrm{NA}_\eta(\balph)$ is denoted by $\mc{A}_\eta(\balph).$ The set of $\eta$-\textit{well approximable points} is defined by
	\[\mc{W}(\eta)=\{\balph\in\R^d:\bm{0}\in\mc{A}_\eta(\balph)\}.\]
	The set $\mc{W}(1/d)^c$ is the set of \textit{badly approximable points} in $\R^d$ (they are the points $\balph$ which satisfy \eqref{eqn.BadDef}). The set $\mc{D}(\eta)$ is defined by
	\[\mc{D}(\eta)=\{\balph\in\R^d:\mc{A}_\eta(\balph)=\R^d\}.\]
	Finally, we use $K_d$ to denote the set of points $\balph\in\R^d$ for which $1,\alpha_1,\ldots, \alpha_d$ are $\Q$-linearly independent.

	\subsection{Diophantine approximation} There are several results from Diophantine approximation which will be used in the proofs of our main results. The first is the following special case of the Khintchine-Groshev theorem.
	\begin{theorem}\label{thm.KhinGros}
		Suppose $d\in\N$, that $\psi:\N\rar (0,\infty)$ is a decreasing function, and that
		\begin{equation}\label{eqn.KGSum}
			\sum_{q=1}^{\infty}(q\psi (q))^d=\infty.
		\end{equation}
		Then, for almost every $\bm{\alpha}\in\R^d$, there are infinitely many $q\in\N$ which satisfy the inequality
		\begin{equation}\label{eqn.KGIneq}
			\|q\bm{\alpha}\|\le q\psi(q).
		\end{equation}
		If the sum in \eqref{eqn.KGSum} converges then, for almost every $\bm{\alpha}\in\R^d$, there are only finitely many $q\in\N$ which satisfy \eqref{eqn.KGIneq}.
	\end{theorem}
	A more general form of this theorem for systems of linear forms can be found in \cite[Section 12.1]{BereDickVela2006}. Next, we will make use of two results of Cassels which are examples of \textit{transference principles} in Diophantine approximation (see \cite[Chapter 5]{Cass1957}). The first of these demonstrates that, roughly speaking, if a point $\balph\in\R^d$ is not too well approximable by rationals, then its orbit modulo $\Z^d$ must become dense in $\R^d/\Z^d$ at a certain rate.
	\begin{lemma}\label{lem.TransPrinc1}
		Suppose that $\balph\in\R^d, \delta>0,$ and $Q\in [1,\infty)$. If 
		\[\min_{0< q\le Q}\|q\balph\|\ge \delta\]
		then the set of points
		\[\left\{q\balph+\bm{p}:|q|\le Q',\bm{p}\in\Z^d\right\}\]
		is $\varrho$-dense in $\R^d,$ where
		\[Q'=\max\left\{Q,~\frac{1}{\delta^{d}}\right\}\quad\text{and}\quad\varrho=\max\left\{\delta,~\frac{1}{Q\delta^{d-1}}\right\}.\]	
	\end{lemma}
	Lemma \ref{lem.TransPrinc1} follows directly from \cite[Chapter 5, Theorem VI]{Cass1957}. The second transference principle establishes a connection between good approximations to $\balph$ and small values (modulo 1) of the linear form determined by the coefficients of $\balph$.
	\begin{lemma}\label{lem.TransPrinc2}
		Suppose that $\balph\in\R^d, \delta\in (0,1),$ and $Q\in [1,\infty)$. If \[\displaystyle \min_{0< q\le Q}\|q\balph\|\le \delta,\] then
		\[\min\left\{\|\bm{r}\balph\|:\bm{r}\in\Z^d, 0<|\bm{r}|\le R\right\}\le\varrho,\]
		where $\displaystyle R=dQ^{1/d}$ and $\displaystyle \varrho=\frac{d\delta}{Q^{1-1/d}}.$
		\end{lemma}
		This lemma follows from the $m=1, n=d$ case of \cite[Chapter 5, Theorem II]{Cass1957}.
		
		Next, following Baker \cite{Bake1977}, we say that a point $\balph\in\R^d$ is \textit{singular} if
		\[\lim_{R\rar\infty}R^d\cdot\min\left\{\|\bm{r}\balph\|:\bm{r}\in\Z^d, 0<|\bm{r}|\le R\right\}=0.\]
		When $d=1$ it is not difficult to show that a number $\alpha$ will be singular if and only if $\alpha\in\Q$. It is obvious from the definitions that, for arbitrary $d\in\N$, the elements of $K_d^c$ are always singular. However, it is a non-trivial fact, proved by Khintchine \cite{Khin1926}, that for $d\ge 2$ there are also singular points which lie in $K_d$. In fact Khintchine's result implies the stronger statement that, for any natural number $d\ge 2$ and real $\omega\ge d$, the set
		\[E(\omega)=\left\{\balph\in K_d:\lim_{R\rar\infty}R^\omega\min\left\{\|\bm{r}\balph\|:\bm{r}\in\Z^d, 0<|\bm{r}|\le R\right\}=0\right\}\]
		is nonempty. Subsequent work by Baker \cite{Bake1977,Bake1992}, Yavid \cite{Yavi1987}, and Rynne \cite{Rynn1990} led to the following more precise estimates for the sizes of these sets.
		\begin{theorem}\label{thm.SingHDBds}
			For $d\in\N,$ $d\ge 2,$ and $\omega>d$, we have that
			\[d-2+\frac{d}{\omega}\le \dim_H(E(\omega))\le d-2+\frac{2d+2}{\omega+1}.\]
		\end{theorem}
		As an application of Lemma \ref{lem.TransPrinc2}, we also have the following result.
		\begin{lemma}\label{lem.SingTrans}
			For $d\in\N, d\ge 2,$ and $\omega\ge 1/d$, we have that
			\[\left\{\balph\in K_d:\lim_{Q\rar\infty}Q^\omega\min_{0<q\le Q}\|q\balph\|=0\right\}\subseteq E(d-1+d\omega).\]
		\end{lemma}
		\begin{proof}
			Suppose that $\balph$ is an element of the set on the left hand side of the above equation. Then, for any $\epsilon>0$, and for all sufficiently large $Q$ (depending on $\epsilon$), we have that
			\[\min_{0<q\le Q}\|q\balph\|\le\frac{\epsilon}{Q^\omega}.\]
			By Lemma \ref{lem.TransPrinc2}, for $R=dQ^{1/d}$ we have that
			\[\min\left\{\|\bm{r}\balph\|:\bm{r}\in\Z^d, 0<|\bm{r}|\le R\right\}\le\frac{d\epsilon}{Q^{1-1/d+\omega}}=\frac{d^{d(1+\omega)}\epsilon}{R^{d-1+d\omega}}.\]
			Since $\epsilon$ can be taken arbitrarily small, this shows that $\balph\in E(d-1+d\omega)$.
		\end{proof}
		
		\subsection{Measure theory and Hausdorff dimension}
		In our proof of Theorem \ref{thm.LebMeasDens} we will use the following higher dimensional version of a lemma due to Cassels \cite[Lemma 9]{Cass1950}.
		\begin{lemma}\label{lem.CassHighD}
			Suppose $d\in\N$ and that $\{C_n\}_{n=1}^\infty$ is a collection of hypercubes in $\R^d$, with faces parallel to coordinate hyperplanes in $\R^d$, and with diameters tending to $0$ as $n\rar\infty.$ If $\epsilon>0$ and if $\{U_n\}_{n=1}^\infty$ is any collection of Lebesgue measurable sets satisfying \[U_n\subseteq C_n ~\text{ and }~ \lambda(U_n)\ge \epsilon\lambda(C_n),\] for all $n\in\N$, then
			\[\lambda\left(\limsup_{n\rar\infty} U_n\right)=\lambda\left(\limsup_{n\rar\infty} C_n\right).\]
		\end{lemma}
		The proof of Lemma \ref{lem.CassHighD} is a straightforward modification of the proof of \cite[Lemma 2]{Gall1961}, the only difference being the use of a multi-dimensional version of the Lebesgue density theorem (with local density defined using sup-norm balls in $\R^d$).
		
		For the lower bounds on Hausdorff dimensions in Theorems \ref{thm.SingDens} and \ref{thm.JarBesDens} we will employ a result known as the \textit{mass distribution principle}. To state this principle concisely, recall that a mass distribution on a set $A\subseteq\R^d$ is a finite, non-zero Borel measure with support contained in $A$.
		\begin{lemma}\label{lem.MassTransPrinc}
			Suppose that $A\subseteq\R^d$ and that $\mu$ is a mass distribution on $A$. If there exist $c,s,\epsilon>0$ with the property that
			\[\mu(B)\le c\cdot\mathrm{diam}(B)^s,\]
			for any set $B$ with diameter $\mathrm{diam}(B)<\epsilon$, then $\dim_H(A)\ge s$.
		\end{lemma}
		The mass distribution principle follows almost directly from the definitions of Hausdorff measure and dimension (see \cite[Section 4.1]{Falc2014}). However, in practice it is a useful way of organizing information in order to establish lower bounds on the dimensions of many naturally occurring self-similar sets.

		\subsection{Algebraic number theory}\label{subsec.AlgNT}
		For our proofs of Theorems \ref{thm.QuadratApps} and \ref{thm.CubicApps} we will need some results from algebraic number theory. For more details the reader may consult \cite{Swin2001}. Let $K\subseteq \R$ be a real algebraic number field with $[K:\Q]=d+1$. Note that, by supposing that $K\subseteq\R$, we are technically fixing one real embedding of the algebraic number field (the one which we call $\sigma_0$ below) and identifying $K$ with its image under this embedding. We write $\mc{O}_K$ for the ring of integers in $K$ and $\NN(\cdot)$ and $\Tr(\cdot)$ for the norm and trace maps from $K$ to $\Q$. 
		
		Suppose that $K$ has $r+1$ distinct embeddings into $\R$ and $2s$ distinct non-real embeddings into $\C$ (so $d+1=r+1+2s$). We label these as $\sigma_0,\ldots,\sigma_d$, with $\sigma_0,\ldots ,\sigma_r$ being the real embeddings and with $\sigma_{r+i}=\overline{\sigma}_{r+s+i}$ for $1\le i\le s$. We further assume that $\sigma_0$ is the identity map on $K$. Identifying $\C$ with $\R^2$, we define the \textit{Minkowski embedding} $\sigma:K\rar\R^{d+1}$ by
		\[\sigma(\alpha)=(\sigma_0(\alpha),\ldots ,\sigma_{r+s}(\alpha))^t,\]
		where the $t$ denotes transpose (we view elements in the image of this map as column vectors). The Minkowski embedding is an injective map and in what follows we will often identify subsets of $K$ with their images under this map.
				
		An additive subgroup $\mc{L}$ of $\R^{d+1}$ is called a \textit{lattice} (using the definition from Lie theory) if it is discrete and if the quotient $\R^{d+1}/\mc{L}$ has a measurable fundamental domain of finite volume. Therefore, we will refer to a subset $\Lambda\subseteq K$ as a \textit{lattice} if $\sigma(\Lambda)$ a lattice in $\R^{d+1}$. For example, it is well known that any fractional ideal in $K$ is a lattice. More generally, if $\alpha_0,\ldots ,\alpha_d\in K$ are $\Q$-linearly independent then the $\Z$-module $\Lambda(\alpha_0,\ldots ,\alpha_d)$ defined by
		\[\Lambda(\alpha_0,\ldots ,\alpha_d)=\alpha_0\Z+\cdots +\alpha_d\Z\]
		is a lattice in $K$.

		The \textit{logarithmic embedding}  is the map $\varphi:K^\times\rar\R^{r+s+1}$ defined by
		\[\varphi(\alpha)=(\log|\sigma_0(\alpha)|,\ldots ,\log |\sigma_{r+s}(\alpha)|)^t.\]
		This map is a homomorphism of the multiplicative group $K^\times$ into $\R^{r+s+1}$, and the image of the subgroup $\mc{O}_K^\times$ is contained in the hyperplane in $\R^{r+s+1}$ with equation
		\begin{equation}\label{eqn.UnitPlane}
			x_0+\cdots +x_r+2x_{r+1}+\cdots +2x_{r+s}=0.
		\end{equation}
		By the Dirichlet unit theorem \cite[\S1.3]{Swin2001} the group $\varphi (\mc{O}_K^\times)$ is a lattice in this hyperplane.
		
		The symmetric bilinear form $\langle\cdot,\cdot\rangle:K\times K\rar\Q$ defined by
		\[\langle \alpha,\beta\rangle= \Tr (\alpha\beta)\]
		is called the \textit{trace form} on $K$. If $\alpha_0,\ldots ,\alpha_d\in K$ are $\Q$-linearly independent then the matrix \[A=\left(\langle\alpha_i,\alpha_j\rangle\right)_{0\le i,j\le d}\]
		has rational entries, and $\det(A)$ is a non-zero rational multiple of the discriminant of $K$. Therefore $A$ is invertible and its inverse has rational entries. Write $A^{-1}=(a_{ij}^*)$ and for $0\le j\le d$ define
		\[\alpha_j^*=\sum_{i=0}^da_{ij}^*\alpha_i.\] 
		Then it is not difficult to see that $\alpha_0^*,\ldots,\alpha_d^*$ are $\Q$-linearly independent elements of $K$ and, for $0\le i,j\le d$,
		\[\langle\alpha_i,\alpha_j^*\rangle=\begin{cases}
			1 &\text{if }i=j,\\
			0 &\text{otherwise}.
		\end{cases}\]
		The lattice $\Lambda^*=\Lambda(\alpha_0^*,\ldots,\alpha_d^*)$ is called the \textit{lattice dual to} $\Lambda=\Lambda(\alpha_0,\ldots,\alpha_d)$, and the basis $\{\alpha_0^*,\ldots,\alpha_d^*\}$ is the \textit{basis dual to} $\{\alpha_0,\ldots,\alpha_d\}$, with respect to the trace form.
		
		If $K$ is a totally real field (i.e. if $s=0$) then the trace form corresponds, via the Minkowski embedding, to the standard inner product on $\R^{d+1}$, and the lattice dual to a lattice $\Lambda=\Lambda(\alpha_0,\ldots ,\alpha_d)$ with respect to the trace form corresponds to the `dual lattice' to $\sigma(\Lambda)$, in the usual sense of this term (i.e. the sense used for lattices in Euclidean space). As a word of warning, if $s>0$ then this is no longer true and the trace form is not an inner product (it is not positive definite). However, in any case, we do have that
		\begin{equation}\label{eqn.BMat1}
			\left(\sigma(\alpha_0^*) \cdots \sigma(\alpha_d^*) \right)=B^{-t},
		\end{equation}
		where $B$ is the $(d+1)\times (d+1)$ matrix defined by
		\begin{equation}\label{eqn.BMat2}
			B=\begin{pmatrix}
				~\textbf{I}_{r+1}	& \rvline & \bm{0} 	\\
				\hline	
				\bm{0} & \rvline &
				\begin{array}{ccccc}
					2 & 0 &\cdots & & 0\\
					0 & -2 & & &\vdots\\
					\vdots &  & \ddots & &\\
					&   &  & 2 & 0\\
					0 & \cdots  &  & 0 & -2
				\end{array}
			\end{pmatrix}\left(\sigma(\alpha_0) \cdots \sigma(\alpha_d) \right).
		\end{equation}
		This observation can be used to show that, for any $s\ge 0$, an equivalent definition of $\Lambda^*$ is
		\begin{equation}\label{eqn.DualLatAltDef}
			\Lambda^*=\left\{\alpha\in K:\langle\alpha,\beta\rangle\in\Z \text{ for all } \beta\in\Lambda\right\}.
		\end{equation}
		
		In what follows it will be useful to identify, for a given lattice $\Lambda\subseteq K$, a full rank subgroup of $\mc{O}_K^\times$ which acts by multiplication as automorphisms on $\Lambda^*$. To this end, we define $\mc{Z}_\Lambda\subseteq\mc{O}_K$ by
		\begin{equation*}
			\mc{Z}_\Lambda=\left\{\gamma\in\mc{O}_K:\gamma\Lambda\subseteq\Lambda\right\}.
		\end{equation*}
		From the definition it is clear that $\mc{Z}_\Lambda$ is a subring of $\mc{O}_K$, and it is not difficult to show that it is also a lattice. Furthermore, for all $\gamma\in \mc{Z}_\Lambda$ and $\alpha\in\Lambda^*,$ and for all $\beta\in\Lambda$, we have from \eqref{eqn.DualLatAltDef} that 
		\[\langle \gamma\alpha,\beta\rangle =\langle \alpha,\gamma\beta\rangle\in\Z.\]
		This shows that
		\[\gamma\Lambda^*\subseteq\Lambda^* ~\text{ for all }~\gamma\in \mc{Z}_\Lambda.\]			
		By definition, $\mc{Z}_\Lambda^\times$ is a subgroup of $\mc{O}_K^\times$, and it turns out that it is also a subgroup of maximal rank. For completeness we include a proof of this fact.
		\begin{lemma}
			If $\Lambda$ is a lattice in $K$ then the logarithmic embedding $\varphi(\mc{Z}_\Lambda^\times)$ is a lattice in the hyperplane in $\R^{r+s+1}$ defined by \eqref{eqn.UnitPlane}.
		\end{lemma}
		\begin{proof}
			Let $n$ be a positive integer with the property that $n\Lambda$ is a sub-lattice of $\mc{O}_K$. Then $\sigma(n\Lambda)$ is a finite index subgroup of $\sigma(\mc{O}_K)$. Note that there are only finitely many such subgroups of a given index, and that multiplication in $K$ by any element of $\mc{O}_K^\times$ (which corresponds via the Minkowski embedding to a volume preserving automorphism of $\sigma(\mc{O}_K)$) permutes them.
			
			Now let $v_1,\ldots v_{r+s}$ be a basis for $\mc{O}_K^\times$. It follows from the above discussion that, for each $1\le i\le r+s$, there is a positive integer $b_i$ with the property that multiplication by $v_i^{b_i}$ fixes $n\Lambda$. Therefore the collection of units $v_1^{b_1},\ldots ,v_{r+s}^{b_{r+s}}$ generates a maximal rank subgroup of $\mc{Z}_\Lambda^\times$, and the result of the lemma follows.
		\end{proof}
		Since $\mc{Z}_\Lambda^\times$ has rank $r+s$ and $\mc{Z}_\Lambda^\times\subseteq\R$, we can write
		\[\mc{Z}_\Lambda^\times=\left\{\pm u_1^{a_1}\cdots u_{r+s}^{a_{r+s}}:a_1,\ldots , a_{r+s}\in\Z\right\},\]
		for some collection of \textit{fundamental units} $u_1,\ldots ,u_{r+s}$ for $\mc{Z}_\Lambda^\times$. Additionally, there is a real number $\kappa>1$ with the property that every closed cube in $\R^{r+s+1}$ of side length $2\log\kappa$ with center on the hyperplane \eqref{eqn.UnitPlane} contains a point of $\varphi(\mc{Z}_\Lambda^\times)$. Let $\kappa_\Lambda>1$ be the infimum of all such numbers $\kappa$. It follows that, for every $t>0$, there is a unit $u\in\mc{Z}_\Lambda^\times$ with
		\[\varphi(u)\in [\log t-\log\kappa_\Lambda,\log t+\log\kappa_\Lambda] \times\prod_{i=1}^{r+s}\left[-\frac{\log t}{d}-\log\kappa_\Lambda, -\frac{\log t}{d}+\log\kappa_\Lambda\right].\]
		The unit $u$ then satisfies
		\begin{equation}\label{eqn.DomUnitDef}
			\kappa_\Lambda^{-1}t\le |u|\le \kappa_\Lambda t,\quad\text{and}\quad\kappa_\Lambda^{-1}t^{-1/d}\le |\sigma_i(u)|\le \kappa_\Lambda t^{-1/d}\quad\text{for}\quad1\le i\le d.
		\end{equation}
		Any unit $u\in\mc{Z}_\Lambda^\times$ satisfying equations \eqref{eqn.DomUnitDef}, for some $t>1$, will be called a \textit{dominant unit}.

		\section{Proof of Theorem \ref{thm.LebMeasDens}}\label{sec.ThmDens}
		We may assume without loss of generality that $0<\eta\le 1/d$ (cf. Table \ref{fig.NewResults}). For each $\bm{n}\in\Z^d, q\in\Z\setminus\{0\}, \bm{\gamma}\in\R^d$, and $\epsilon>0$ let
		\[\mc{E}_{\bm{n},q}(\bm{\gamma},\epsilon)=\left\{\bm{x}\in\bm{n}+[0,1)^d : 0<\left||q|^{\eta}\left(q\bm{x}-\bm{p}\right)-\bm{\gamma}\right|<\epsilon \text{ for some } \bm{p}\in\Z^d\right\}.\]
		Then define
		\[\mc{E}_{\bm{n},\infty}(\bm{\gamma},\epsilon)=\limsup_{|q|\rar\infty} \mc{E}_{\bm{n},q}(\bm{\gamma},\epsilon)\quad\text{and}\quad \mc{E}_{\infty}(\bm{\gamma},\epsilon)=\bigcup_{\bm{n}\in\Z^d}\mc{E}_{\bm{n},\infty}(\bm{\gamma},\epsilon).\]
		
		From our definitions we have that, for all $\bm{\alpha}\in\R^d$, and for all $\bm{\gamma}\in\R^d$,
		\begin{equation}\label{eqn.AccumPtSets1}
			\bm{\gamma}\in\mc{A}_\eta(\bm{\alpha})\quad\text{ if and only if }\quad \bm{\alpha}\in\bigcap_{\epsilon>0}\mc{E}_{\infty}(\bm{\gamma},\epsilon).
		\end{equation}
		In view of this, what we are aiming to prove will follow easily from the statement that, for every $\bm{n}\in\Z^d,\bm{\gamma}\in\R^d$, and $\epsilon>0$,
		\begin{equation}\label{eqn.FullMeas1}
			\lambda\left(\mc{E}_{\bm{n},\infty}(\bm{\gamma},\epsilon)\right)=1.
		\end{equation}
		If $\bm{\gamma}=\bm{0}$ then this holds by Theorem \ref{thm.KhinGros}, so we can assume in the next part of the argument that $\bm{\gamma}\not=\bm{0}$.
		
		Suppose for simplicity in what follows that $q\ge 1$. To see why \eqref{eqn.FullMeas1} is true, first note that
		\begin{equation*}
			\overline{\mc{E}}_{\bm{n},q}(\bm{\gamma},\epsilon)=\bigcup_{\bm{p}\in\Z^d}\left(\left(\prod_{i=1}^d[n_i,n_i+1]\right)\bigcap\left(\prod_{i=1}^d\overline{B}\left(\frac{p_i}{q}+\frac{\gamma_i}{q^{1+\eta}},\frac{\epsilon}{q^{1+\eta}}\right)\right)\right).
		\end{equation*}
		This set differs from $\mc{E}_{\bm{n},q}(\bm{\gamma},\epsilon)$ only possibly at the centers and along the boundaries of the intervals involved, which will not affect the Lebesgue measures of the corresponding limsup sets obtained by letting $q\rar\infty$. Also, when $q$ is large enough, $\epsilon/q^{1+\eta}<1/2q$, so that the closed balls in the second intersection do not overlap.
		
		
		Next, note that the set $\overline{\mc{E}}_{\bm{n},q}(\bm{\gamma},2|\bm{\gamma}|)$ contains the set \[\mc{E}_{\bm{n},q}'(\bm{\gamma})=\bigcup_{\bm{p}\in\Z^d}\left(\left(\prod_{i=1}^d[n_i,n_i+1)\right)\bigcap\left(\prod_{i=1}^dB\left(\frac{p_i}{q},\frac{|\bm{\gamma}|}{q^{1+\eta}}\right)\right)\right),\]
		which implies that
		\begin{align*}
			\lambda\left(\mc{E}_{\bm{n},\infty}(\bm{\gamma},2|\bm{\gamma}|)\right)
			&\ge\lambda\left(\limsup_{q\rar\infty}\mc{E}_{\bm{n},q}'(\bm{\gamma})\right)\\
			&=\lambda\left(\left\{\bm{x}\in\bm{n}+[0,1)^d : \|q\bm{x}\|\le |\bm{\gamma}|q^{-\eta} \text{ for infinitely many } q\in\N\right\}\right).
		\end{align*}
		Therefore (since $\bm{\gamma}\not=\bm{0}$) by Theorem \ref{thm.KhinGros} we have that
		\begin{equation}\label{eqn.FullMeas2}
			\lambda\left(\mc{E}_{\bm{n},\infty}(\bm{\gamma},2|\bm{\gamma}|)\right)=1.
		\end{equation}
		
		We now wish to apply Lemma \ref{lem.CassHighD} but first we must deal with a small technicality, which is that the sets $\overline{\mc{E}}_{\bm{n},q}(\bm{\gamma},\epsilon)$ may not be finite unions of hypercubes. The issue is that along the boundaries of $\bm{n}+[0,1]^d$ these sets may contain connected components with arbitrarily large ratio of diameter to inradius (care must be taken since the Lebesgue density theorem, a key ingredient in the proof of Lemma \ref{lem.CassHighD}, does not hold in general when density is defined using collections of shapes like this). However, since $\epsilon/q^{1+\eta}<1/2q$ for large enough $q$, the number of connected components of $\overline{\mc{E}}_{\bm{n},q}(\bm{\gamma},\epsilon)$ which intersect the boundary of $\bm{n}+[0,1]^d$ is $\ll q^{d-1}$. Each of these connected components has volume
		$\ll q^{-d(1+\eta)}$. Since $\eta>0$ we have that
		\[\sum_{q\in\N}\frac{q^{d-1}}{q^{d(1+\eta)}}<\infty,\]
		so we deduce from the Borel-Cantelli lemma that removing these components from each of the sets $\overline{\mc{E}}_{\bm{n},q}(\bm{\gamma},\epsilon)$ does not affect the measure of the limsup set as $q\rar\infty$. However, after removing them we are left with sets which are finite unions of hypercubes in $\R^d$, with faces parallel to coordinate hyperplanes. Applying Lemma \ref{lem.CassHighD} to the resulting sets, we see that, for $\bm{\gamma}\not=\bm{0}$, \eqref{eqn.FullMeas1} follows from \eqref{eqn.FullMeas2}.
		
		Finally, by \eqref{eqn.AccumPtSets1} we have that
		\begin{align*}
			\bigcup_{\bm{\gamma}\in\Q^d}\left\{\bm{\alpha}\in\R^d : \bm{\gamma}\notin\mc{A}_\eta(\bm{\alpha})\right\}=	\bigcup_{\bm{\gamma}\in\Q^d}\left(\bigcap_{k\in\N}\mc{E}_{\infty}(\bm{\gamma},2^{-k})\right)^c=	\bigcup_{\bm{\gamma}\in\Q^d}\bigcup_{k\in\N}\mc{E}_{\infty}(\bm{\gamma},2^{-k})^c.
		\end{align*}
		The right hand side is a countable union of null sets, so we have that $\Q^d\subseteq \mc{A}_\eta(\bm{\alpha}),$ for almost every $\bm{\alpha}\in\R^d$. Since $\mc{A}_\eta(\balph)$ is a closed set, this concludes the proof of Theorem \ref{thm.LebMeasDens}.

		\section{Proof of Theorem \ref{thm.SingDens}}\label{sec.SingThm}
		We divide the proof of Theorem \ref{thm.SingDens} into two parts.
		
		\noindent \textbf{Upper bound for Hausdorff dimension:} For $\omega>0$, define
		\[S(\omega)=\left\{\balph\in K_d:\lim_{Q\rar\infty}Q^\omega\min_{0<q\le Q}\|q\balph\|=0\right\}.\]
		In order to prove the upper bound from \eqref{eqn.HDSingUpBd}, we first establish the following lemma.
		\begin{lemma}\label{lem.NonDenseSing}
			For $d\in\N$ and $0<\eta<1/d$, we have that
			\[\mc{D}(\eta)^c\setminus K_d^c\subseteq S(\omega),\]
			for all $\omega<1/(d\eta+d-1)$.
		\end{lemma}
		\begin{proof}
			Suppose that $\omega$ satisfies the required bound and, without loss of generality, assume also that $\omega>1/d.$ If $\balph\notin S(\omega)\cup K_d^c$, then there are a constant $c_0>0$ and infinitely many $Q\in\N$ for which
			\[\min_{0<q\le Q}\|q\balph\|\ge \frac{c_0}{Q^\omega}.\]
			The transference principle from Lemma \ref{lem.TransPrinc1} then guarantees that
			\begin{equation}\label{eqn.TransfDenseSet}
				\left\{q\balph+\bm{p}:|q|\le Q',\bm{p}\in\Z^d\right\}
			\end{equation}
			is $\varrho$-dense in $\R^d$ where, using our assumption that $\omega>1/d$ and taking $Q$ sufficiently large,
			\[Q'=\frac{Q^{d\omega}}{c_0^d}\quad\text{and}\quad\varrho=\frac{1}{c_0^{d-1}Q^{1-(d-1)\omega}}.\]
			Translating the set in \eqref{eqn.TransfDenseSet} and relabeling, we have that, for any $Q_0\in\R$, the set
			\[\left\{q\balph+\bm{p}:Q_0-Q_1 <q\le Q_0, \bm{p}\in\Z^d\right\}\]
			is $(c_1/Q_1^{\eta'})$-dense, where
			\[Q_1=c_2Q^{d\omega},\quad\eta'=\frac{1}{d\omega}-\frac{d-1}{d},\]
			and $c_1$ and $c_2$ are positive constants which depend only on $c_0$ and $d$. Note that the assumption that $\omega<1/(d\eta+d-1)$ guarantees that $\eta'>\eta$.
			
			Now, choose $\bm{\gamma}\in\R^d\setminus\{\bm{0}\}$ with $\bm{\gamma}\notin\mathrm{NA}_\eta(\balph).$ We wish to show that, for any $\epsilon>0$, there exists a $q\in\N$ for which 
			\begin{equation}\label{eqn.AccPtForm3}
				\left\|q\bm{\alpha}-\frac{\bm{\gamma}}{q^{\eta}}\right\|<\frac{\epsilon}{q^{\eta}}.
			\end{equation}
			In order to do this, first set
			\[\rho=\left(1+\frac{\epsilon}{|\bm{\gamma}|}\right)^{1/\eta}\quad\text{and}\quad c_\epsilon=\frac{(\rho-1)^\eta(\rho^\eta-1)|\bm{\gamma}|}{\rho^\eta},\]
			choose $Q$ large enough so that
			\[\frac{c_1}{Q_1^{\eta'}}<\frac{c_\epsilon}{Q_1^\eta},\]
			and then set
			\[Q_0=\frac{\rho Q_1}{\rho-1}.\]
			Then $Q_0-Q_1=Q_0/\rho$ and so, as $q$ varies from $Q_0-Q_1$ to $Q_0$, the $i$th component of $\bm{\gamma}/q^\eta$ changes by at most
			\[\frac{(\rho^\eta-1)|\gamma_i|}{Q_0^\eta}<\frac{\epsilon}{Q_0^\eta}=\frac{c_\epsilon}{Q_1^\eta}.\]
			It therefore follows that there is a $Q_0-Q_1<q\le Q_0$ for which \eqref{eqn.AccPtForm3} is satisfied. Since $\epsilon$ is arbitrary, this shows that $\balph\in\mc{D}(\eta)$, which completes the proof of the lemma.
		\end{proof}
		Note that the $d=1$ case of this lemma (together with the fact that $S(1)=\emptyset$ when $d=1$) implies the $d=1$ case of Theorem \ref{thm.SingDens}.
	
		Combining the result of Lemma \ref{lem.NonDenseSing} with that of Lemma \ref{lem.SingTrans}, we have that, for $d\ge 2,$ $0<\eta<1/d$, and $\omega<1/(d\eta+d-1)$,
		\[\mc{D}(\eta)^c\setminus K_d^c\subseteq E(d-1+d\omega).\]
		By Theorem \ref{thm.SingHDBds} this implies that
		\[\dim_H(\mc{D}(\eta)^c\setminus K_d^c)\le d-2+\frac{2d+2}{d+d\omega},\]
		and letting $\omega$ tend to $1/(d\eta+d-1)$ gives the bound \eqref{eqn.HDSingUpBd} in the statement of Theorem \ref{thm.SingDens}.
		
		\noindent \textbf{Lower bound for Hausdorff dimension:} To prove the lower bound \eqref{eqn.HDSingLowBd}, we will show that, for $d\ge 2$ and $0<\eta<1/d$, the set $\mc{D}(\eta)^c\setminus K_d^c$ contains a subset of $E(1+1/\eta)$ with Hausdorff dimension bounded below by the required amount. To make this precise, for $\omega\ge d$ define
		\[E_+(\omega)=\left\{\balph\in K_d:\lim_{R\rar\infty}R^\omega\min_{}\left\{\|\bm{r}\balph\|:\bm{r}\in\Z^d\setminus\{\bm{0}\},~ 0\le r_1,\ldots, r_d\le R\right\}=0\right\}.\]
		First we will prove the following lemma, using a modification of the proof of \cite[Chapter 5, Theorem XV]{Cass1957}.
		\begin{lemma}\label{lem.Sing+Cont}
			For $d\ge 2$ and $0<\eta<1/d$, we have that
			\[E_+(1+1/\eta)\subseteq \mc{D}(\eta)^c\setminus K_d^c.\]
		\end{lemma}
		\begin{proof}
		Suppose by way of contradiction that $\balph\in E_+(1+1/\eta)\cap \mc{D}(\eta)$. Choose $\epsilon$ so that
		\[\epsilon<\frac{1}{2(2d)^{1+1/\eta}},\]
		and let $R_0$ be large enough so that, for every $R\ge R_0$, there exists $\bm{r}\in \Z^d\setminus\{\bm{0}\}$ with $0\le r_1,\ldots, r_d\le R$ and
		\begin{equation}\label{eqn.SingLemInc1}
			\|\bm{r}\balph\|\le \frac{\epsilon}{R^{1+1/\eta}}.
		\end{equation}
		
		Since $(1,1,\ldots,1)\in\mc{A}_\eta(\balph)$, we can find $q\in\Z$ with $|q|^\eta>2dR_0$, such that
		\[	0<\max_{1\le i\le d}\left\|q\alpha_i-\frac{1}{|q|^{\eta}}\right\|<\frac{\epsilon}{|q|^{\eta}}.\]
		If we then take $R=|q|^\eta/(2d)$ and $\bm{r}$ as above, non-zero and with non-negative coordinates, and satisfying $|\bm{r}|\le |q|^\eta/(2d)$ and \eqref{eqn.SingLemInc1}, we have that
		\begin{align*}
		\frac{r_1+\cdots+r_d}{|q|^\eta}&=\left\|\frac{r_1+\cdots +r_d}{|q|^\eta}\right\|\\
		&=\left\|\sum_{i=1}^d r_i\left(q\alpha_i-\frac{1}{|q|^\eta}\right)-q(\bm{r}\balph)\right\|\\
		&\le\sum_{i=1}^dr_i\left\|q\alpha_i-\frac{1}{|q|^\eta}\right\|+|q|\|\bm{r}\balph\|\\
		&<\frac{\epsilon(r_1+\cdots+r_d)}{|q|^\eta}+\frac{\epsilon (2d)^{1+1/\eta}}{|q|^\eta}.
		\end{align*}
		By our choice of $\epsilon$, this is a contradiction. Therefore the statement of the lemma is established.
		\end{proof}		
		The lower bound in the statement of Theorem \ref{thm.SingDens} follows from combining Lemma \ref{lem.Sing+Cont} with the following result.
		\begin{lemma}\label{lem.Sing+DimLowBd}
			For $d\ge 2$ and $\omega\ge d$, we have that
			\[\dim_H(E_+(\omega))\ge \frac{d}{\omega}.\]
		\end{lemma}
		\begin{proof}
		We will explain how to derive this result using a technical modification in the proof of the lower bound for $\dim_H(E(\omega))$ in \cite{Bake1977}, together with an application of the mass distribution principle. The key modification is to the part of the proof of \cite[Lemma 3]{Bake1977} at the top of page 382 of that article, where the author appeals to a special case of Siegel's lemma to deduce that, if $a_1,\ldots ,a_{d-1},c$ are non-negative integers bounded by $N$, with $c>0,$ then there exist $\bm{y}\in\Z^{d-1}$ and $q\in\Z$ such that
		\[a_1y_1+\cdots+a_{d-1}y_{d-1}+cq=0\quad\text{and}\quad 0<|\bm{y}|\le (dN)^{1/(d-1)}.\]
		Here we would like to choose $\bm{y}$ so that all of its components are non-negative. Unfortunately, it is not possible to do this in general, while maintaining good control over the size of $|\bm{y}|$. Therefore we settle (at the cost of some loss in our final lower bound for the Hausdorff dimension) for the choice
		\[\bm{y}=(c,c,\ldots ,c)\quad\text{and}\quad q=-(a_1+\cdots +a_{d-1}),\]
		which satisfies the much weaker bound $|\bm{y}|\le N$. Then, by the same argument in the proof of \cite[Lemma 3]{Bake1977}, but with $N=X$, we obtain the following result.		
		\begin{itemize}[itemsep=10bp,parsep=10bp, topsep=0bp]
			\item Suppose that $\omega_0> d$, that $C\in [0,1)^d$ is a cube with center $\bm{\Phi}$ satisfying the equation
			\begin{equation}\label{eqn.SingLemHyperplane}
				\bm{x}\bm{\Phi}+r=0,
			\end{equation}
			where $\bm{x}\in\Z^d, r\in\Z$, and $|\bm{x}|\le A$. For all sufficiently large $X$ (depending on $C, A,$ and $\omega$), there are points $\balph_1,\ldots ,\balph_t$, with $\bm{x}\balph_i+r=0$ for each $i$, and pairwise disjoint cubes $B_1,\ldots ,B_t$ in $C$ such that, for each $1\le i,j\le t$,
			\begin{itemize}[itemsep=5bp,parsep=10bp, topsep=0bp]
				\item[(i)] $\displaystyle |\balph_i-\balph_j|\ge X^{-d/(d-1)}$ for $i\not=j$,
				\item[(ii)] The center $\bm{\beta}_i$ of each $B_i$ satisfies
				$\displaystyle |\balph_i-\bm{\beta}_i|\le 1/(2dAX^{\omega_0}),$
				\item[(iii)] Each $B_i$ has side length $1/(2d^2AX^{\omega_0})$,
				\item[(iv)] For all $\bm{\theta}\in B_i$, we have $\|\bm{x}\bm{\theta}\|<X^{-\omega_0}$,
				\item[(v)] $B_i$ does not intersect any plane with equation 
				\[\bm{z}\bm{\theta}+s=0,\]
				with $\bm{z}\in\Z^d, s\in\Z$, and $|\bm{z}|\le A$,
				\item[(vi)] Each $\bm{\beta}_i$ lies on a plane with equation
				\[\bm{y}^{(i)}\bm{\theta}+q_i=0,\]
				with $\bm{y}^{(i)}\in\Z^d\setminus\{\bm{0}\}, q_i\in\Z$, and $0\le\bm{y}^{(i)}_k\le X$ for $1\le k\le d$, and
				\item[(vii)] $t\gg \mathrm{diam}(C)^{d-1}X^d$, where the implied constant depends only on $d$.
			\end{itemize}
		\end{itemize}
		One key difference between what we are presenting here and the lemma cited above is in point (vi), where we now have that all of the components of the vectors $\bm{y}^{(i)}$ are non-negative. Using this result, we build a Cantor set $\mc{E}$ and a probability mass distribution $\mu$ supported on $\mc{E}$, by starting with the unit cube and successively dividing each of the cubes constructed at each stage into subcubes, distributing the mass of each cube evenly among the subcubes, and then taking the weak-$^\ast$ limit as this process tends to infinity. Note that condition (vi) allows us to iterate this construction. Conditions (iv) and (v) then guarantee that $\mc{E}\subseteq E_+(\omega_1),$ for all $\omega_1<\omega_0$.
		
		
		Now we apply the mass distribution principle. At the $k$th level of our construction of $\mc{E}$ we write $A=Q_{k-1}$ and $X=Q_k$. Suppose in what follows that $k\ge 4.$ Each cube from the $(k-1)$st level has diameter \[\gg\frac{1}{Q_{k-2}Q_{k-1}^{\omega_0}},\]
		so by (vii) the number $m_k$ of cubes constructed at the $k$th level, from each $(k-1)$st level cube, satisfies
		\[m_k\gg \frac{Q_k^d}{Q_{k-2}^{d-1}Q_{k-1}^{\omega_0(d-1)}}.\]
		Furthermore, by (i) and (ii), these cubes are separated by
		\[\epsilon_k\gg Q_k^{-d/(d-1)}.\]
		
		Suppose that $B$ is any set in $\R^d$ with $\epsilon_k\le \mathrm{diam}(B)<\epsilon_{k-1}$. Then $B$ can intersect at most one cube from the $(k-1)$st level of our construction. The number of cubes from the $k$th level which it intersects is bounded above by $m_k$ and, by virtue of (i), (ii), and the fact that the $\balph_i$ all lie on the hyperplane \eqref{eqn.SingLemHyperplane} (see the proof of \cite[Lemma 4]{Bake1977} for details), this number is also
		\[\ll \left(\frac{\mathrm{diam}(B)}{\epsilon_k}\right)^{d-1}.\]
		Therefore, the number of $k$th level cubes which intersect $B$ is
		\[\ll \min\left(m_k,\left(\frac{\mathrm{diam}(B)}{\epsilon_k}\right)^{d-1}\right)\le \frac{m_k^{1-s/(d-1)}\mathrm{diam}(B)^s}{\epsilon_k^s},\]
		for any $s\in[0,d-1]$. For our mass distribution $\mu$ this means that
		\[\frac{\mu(B)}{\diam(B)^s}\ll\frac{1}{m_1\cdots m_{k-1}\left(m_k^{1/(d-1)}\epsilon_k\right)^s}.\]
		We would like to choose $s$ so that the right hand side above is bounded above by a fixed constant. Therefore consider
		\begin{align*}
		&\frac{\log (m_1\cdots m_{k-1})}{-\log\left(m_k^{1/(d-1)}\epsilon_k\right)}\\
		&\hspace*{40bp}\gg\frac{d\log Q_{k-1}-(d-1)\log(Q_1\cdots Q_{k-3})-(\omega_0(d-1)-d)\log(Q_1\cdots Q_{k-2})}{\omega_0\log Q_{k-1}+\log Q_{k-2}}.
		\end{align*}
		By choosing $Q_k$ large enough, depending on $Q_{k-1}$,  at each stage of our construction, we can ensure that the quantity above tends to $d/\omega_0$ as $k\rar\infty$. This guarantees that, as long as $s<d/\omega_0$, we have that
		\[\mu(B)\ll\diam(B)^s.\]
		The mass distribution principle (Lemma \ref{lem.MassTransPrinc}) then implies that
		\[\dim_H(\mc{E})\ge\frac{d}{\omega_0}.\]
		Taking the limit as $\omega_0\rar\omega^+$ completes the proof of Lemma \ref{lem.Sing+DimLowBd}.
		\end{proof}

		\section{Proof of Theorem \ref{thm.JarBesDens}}\label{sec.JarBesThm}
		In this section we suppose that $d\in\N$ and that $\eta>1/d$. Note that, since $\mc{D}(\eta)\subseteq\mc{W}(\eta),$ we have from the Jarn\'{i}k-Besicovitch theorem that
		\[\dim_H(\mc{D}(\eta))\le \frac{d+1}{1+\eta}.\]
		Therefore, to establish Theorem \ref{thm.JarBesDens}, it suffices to prove the lower bound for the Hausdorff dimension. For this we will make use of the following result, taken from \cite[Lemma 2]{Bake1977}.
		\begin{lemma}\label{lem.RatPtCount}
		Suppose that $C$ is a cube in $\R^d$ and that $U$ is a closed subset of $E$ with $d$-dimensional Lebesgue measure zero. Then, for all sufficiently large $Q$ (depending on $C$ and $U$), there are rational points $\bm{\beta}_1,\ldots,\bm{\beta}_t$ in $C\setminus U$ satisfying the following properties:
		\begin{itemize}[itemsep=5bp,parsep=10bp, topsep=0bp]
			\item[(i)] For each $i$, $\bm{\beta}_i=(a_{1,i}/q_i,\ldots ,a_{d,i}/q_i)$, with $a_{1,i},\ldots,a_{d,i},q_i\in\Z,$ and $1 \le q_i\le Q$,
			\item[(ii)] $\displaystyle |\bm{\beta}_i-\bm{\beta}_j|\ge Q^{-(d+1)/d}$ whenever $i\not=j$, and
			\item[(iii)] $t\ge c_0\hspace*{2bp}\diam (C)^dQ^{d+1}$, for some constant $c_0$ depending only on $d$.
		\end{itemize}
		\end{lemma}
		We will use this lemma to construct a Cantor set $\mc{C}\subseteq\mc{D}(\eta)$ which has maximum possible Hausdorff dimension. Let $\{\bm{\gamma}_i\}_{i\in\N}=\Q^d$ be any enumeration of the rationals in $\R^d$, and define a sequence of positive integers $\{i_k\}_{k\in\N}$ by successively concatenating 1, followed by 1,2, followed by 1,2,3, and so on, so that
		\[(i_1,i_2,\ldots)=(1,1,2,1,2,3,1,2,3,4,\ldots).\]
		Our Cantor set will be constructed in levels, starting from the 0th level and working upward. Each level will consist of cubes of the same diameter, and each new level will be constructed by subdividing each of these cubes into subcubes. For each $k\ge 0$ we will also define an integer $Q_k\ge 1$ with the property that every point $\balph$ which lies in the $k$th level of the Cantor set, for $k\in\N$, satisfies
		\begin{equation}\label{eqn.Thm3AccumPtCond}
		0<\left\|q\bm{\alpha}-\frac{\bm{\gamma}_{i_k}}{q^{\eta}}\right\|<\frac{1}{q^\eta\log q},
		\end{equation}
		for some $Q_{k-1}< q\le Q_k$. This guarantees that, for all $\balph\in\mc{C}$, we have that $\Q^d\subseteq\mc{A}_\eta(\balph)$. Since the set $\mc{A}_\eta(\balph)$ is closed, this implies that $\mc{C}\subseteq\mc{D}(\eta).$
		
		For the 0th level of our Cantor set we start with the cube $[0,1)^d$, and we set $Q_0=1$. Then for $k\in\N$, we construct the $k$th level of our Cantor set as follows. Let $C_1,\ldots ,C_s$ be the cubes from level $k-1$ and write $\delta_{k-1}$ for their common diameter. Choose $R$ large enough so that 
		\begin{equation}\label{eqn.Thm3CubeShrink}
			R>Q_{k-1}\quad\text{and}\quad \frac{|\bm{\gamma}_{i_k}|}{R^\eta}<\frac{\delta_{k-1}}{3}.
		\end{equation}
		For each $1\le r\le s$ let $C_r'$ be the cube with the same center as $C_r$ and with diameter $\delta_{k-1}/3$, and let $U_r$ be the set of all rational points in $C_r'$ with denominators (least common multiples of the denominators of the reduced fraction components) less than or equal to $R$. Applying Lemma \ref{lem.RatPtCount} to each of the pairs $C_r'$ and $U_r$, we can choose $Q_k$  large enough (independently of $r$) so that each cube $C_r'$ contains at least
		\begin{equation}\label{eqn.Thm3m_k}
			m_k=c_0\delta_{k-1}^dQ_k^{d+1}
		\end{equation}
		rational points, with denominators lying in $(R,Q_k]$, which are separated by at least
		\begin{equation}\label{eqn.Thm3eps_k}
			\epsilon_k=Q_k^{-(d+1)/d}.
		\end{equation}
		Now suppose that
		\[\bm{\beta}_1=\left(\frac{a_{1,1}}{q_1},\ldots,\frac{a_{d,1}}{q_1}\right),\ldots, ~\bm{\beta}_t=\left(\frac{a_{1,t}}{q_t},\ldots,\frac{a_{d,t}}{q_t}\right)\]
		are the rational points constructed in this way, which lie in one of the cubes $C_r'$. Then at level $k$ we divide the cube $C_r$ into the subcubes
		\[B\left(\bm{\beta}_1-\frac{\bm{\gamma}_{i_k}}{q_1^{1+\eta}},\frac{1}{Q_k^{1+\eta}\log Q_k}\right),\ldots , ~B\left(\bm{\beta}_t-\frac{\bm{\gamma}_{i_k}}{q_t^{1+\eta}},\frac{1}{Q_k^{1+\eta}\log Q_k}\right).\]
		We make the analogous subdivision for each of the cubes $C_r,~1\le r\le s.$ By \eqref{eqn.Thm3CubeShrink}, each of the new cubes is contained in a level $k-1$ cube, and any $\balph$ which lies in one of the $k$th level cubes satisfies \eqref{eqn.Thm3AccumPtCond} for some $Q_{k-1}<q<Q_k$. Note also that
		\begin{equation}\label{eqn.Thm3del_k}
			\delta_k\gg\frac{1}{Q_k^{1+\eta}\log Q_k}.
		\end{equation}
		
		Now we apply the mass distribution principle. Let $\mu$ be the probability mass distribution constructed using the above Cantor set, by subdividing the measure assigned to each cube evenly among each of next level subcubes which it contains. If $B$ is any subset of $\R^k$ with $\epsilon_k\le \diam(B)<\epsilon_{k-1}$ then it intersects at most
		\[\ll\min\left(m_k,\left(\frac{\diam(B)}{\epsilon_k}\right)^d\right)\le\frac{m_k^{1-s/d}\diam(B)^s}{\epsilon_k^s}\]
		level $k$ cubes, for any $s\in [0,d]$. This implies that
		\[\frac{\mu(B)}{\diam(B)^s}\ll\frac{1}{m_1\cdots m_{k-1}\left(m_k^{1/d}\epsilon_k\right)^s}.\]
		Using \eqref{eqn.Thm3m_k}-\eqref{eqn.Thm3del_k} we have that
		\begin{align*}
			&\frac{\log (m_1\cdots m_{k-1})}{-\log\left(m_k^{1/d}\epsilon_k\right)}\\
			&\hspace*{10bp}\gg\frac{(d+1)\log Q_{k-1}-(d\eta-1)\log(Q_1\cdots Q_{k-2})-d\log(\log(Q_1)\cdots \log(Q_{k-2}))}{(1+\eta)\log Q_{k-1}+\log\log Q_{k-1}}.
		\end{align*}
		By returning to our Cantor set construction, if necessary, we can ensure that the integers $Q_k$ are chosen so that the right hand side here tends to $(d+1)/(1+\eta)$ as $k$ tends to infinity. Then, as we saw in the previous section, the mass distribution principle guarantees that
		\[\dim_H(\mc{C})\ge\frac{d+1}{1+\eta}.\]
		This therefore completes the proof of Theorem \ref{thm.JarBesDens}.

		\section{Analysis of normalized approximations to algebraic numbers}\label{sec.NormAlg}
		In this section we will use the notation and background material introduced in Section \ref{subsec.AlgNT}. We would like to re-emphasize that many of the building blocks of the results which we are collecting together in this and the following two sections appear in various forms in work of Cassels and Swinnerton-Dyer \cite{CassSwin1955}, Peck \cite{Peck1961}, Adams \cite{Adam1969a,Adam1969b,Adam1971}, Chevallier \cite{Chev2002}, Ito and Yasutomi \cite{ItoYasu2006}, and  Hensley \cite[Chapter 7]{Hens2006}.
		
		Suppose that $1,\alpha_1,\ldots,\alpha_d$ form a $\Q$-basis for an algebraic number field $K\subseteq \R$, and set $\Lambda=\Lambda(1,\alpha_1,\ldots,\alpha_d)$. What we aim to show
		can be loosely summarized in the following two statements:
		\begin{itemize}[itemsep=10bp,parsep=10bp, topsep=10bp]
			\item[(i)] All $(1/d)$-normalized approximations to $\balph$ with sup-norm at most $C>0$ can be associated in a natural way with an algebraic number of the form $su\in\Lambda^*$, where $s$ is taken from a finite set $S_C\subseteq\Lambda^*$ and $u\in\mc{Z}_\Lambda^\times$ is a dominant unit (i.e. one satisfying \eqref{eqn.DomUnitDef}, for some $t>1)$.
			\item[(ii)] If $s\in\Lambda^*$ then numbers of the form $su$, where $u$ runs through all dominant units in $\mc{Z}_\Lambda^\times$, give rise to normalized approximations which lie near a dilation of a fixed surface in $\R^d$, homothetically scaled about the origin by a factor of $|\NN(s)|^{1/d}$.
		\end{itemize}
		To make statement (i) precise, we have the following lemma.
		\begin{lemma}\label{lem.NormAppChar1}
			Suppose that $\balph\in\R^d$ and that $1,\alpha_1,\ldots,\alpha_d$ form a $\Q$-basis for an algebraic number field, and set $\Lambda=\Lambda(\alpha_0,\alpha_1,\ldots,\alpha_d)$, with $\alpha_0=1$. For every $C>0$, there is a finite set $S_C\subseteq\Lambda^*$ with the property that, if $q\in\N$ and $\bm{p}\in\Z^d$ satisfy
			\begin{equation}\label{eqn.NormAppCBnd}
				|q\balph-\bm{p}|\le\frac{C}{q^{1/d}}
			\end{equation}
			then 
			\[q\alpha_0^*+p_1\alpha_1^*+\cdots +p_d\alpha_d^*=su,\]
			for some $s\in S_C$ and dominant unit $u\in\mc{Z}_\Lambda^\times$.
		\end{lemma}
		\begin{proof}
			Let $B$ be the matrix defined by \eqref{eqn.BMat2}, so that \eqref{eqn.BMat1} holds. Write
			\[\alpha^*=q\alpha_0^*+p_1\alpha_1^*+\cdots +p_d\alpha_d^*.\]
			Then we have that
			\begin{align*}
				\sigma(\alpha^*)&=B^{-t}(q,p_1,\ldots ,p_d)^t\\
				&=B^{-t}\left((0,p_1-q\alpha_1,\ldots ,p_d-q\alpha_d)^t+q(1,\alpha_1,\ldots ,\alpha_d)^t\right)\\
				&=\sum_{i=1}^d(p_i-q\alpha_i)\sigma(\alpha_i^*)+q\left((1,\alpha_1,\ldots  ,\alpha_d)B^{-1}\right)^t.
			\end{align*}
			Since
			\[(1,\alpha_1,\ldots  ,\alpha_d)=(1,0,\ldots ,0)(\sigma(\alpha_0)\sigma(\alpha_1)\cdots\sigma(\alpha_d)),\]
			it follows from the definition of $B$ that
			\[(1,\alpha_1,\ldots  ,\alpha_d)B^{-1}=(1,0,\ldots,0).\]
			Therefore we have that
			\[\sigma(\alpha^*)=\sum_{i=1}^d(p_i-q\alpha_i)\sigma(\alpha_i^*)+(q,0,\ldots ,0)^t.\]
			By the assumption that \eqref{eqn.NormAppCBnd} holds, this implies that
			\[|\alpha^*-q|\le \frac{C(\balph)}{q^{1/d}},\]
			and that
			\[|\sigma_i(\alpha^*)|\le\frac{C(\balph)}{q^{1/d}}\quad\text{for}\quad 1\le i\le d,\]
			where
			\[C(\balph)=C\sum_{i=1}^d|\sigma(\alpha_i^*)|.\]
			
			By the discussion preceding \eqref{eqn.DomUnitDef}, there is a dominant unit $u\in\mc{Z}_\Lambda^\times$ satisfying
			\[	\kappa_\Lambda^{-1}q\le |u|\le \kappa_\Lambda q,\quad\text{and}\quad \kappa_\Lambda^{-1}q^{-1/d}\le |\sigma_i(u)|\le \kappa_\Lambda q^{-1/d}\quad\text{for}\quad1\le i\le d.\]
			We then have that
			\[|\sigma(\alpha^*u^{-1})|\le \kappa_\Lambda\left(1+C(\balph)(1+q^{-1-1/d})\right).\]
			Now define
			\[S_C=\left\{s\in\Lambda^*:|\sigma(s)|\le \kappa_\Lambda\left(1+C(\balph)(1+q^{-1-1/d})\right)\right\}.\]
			The fact that $\Lambda^*$ is a lattice guarantees that $S_C$ is a finite set. Finally, since $\sigma$ is injective and $\sigma(\alpha^*u^{-1})\in\sigma(S_C)$, we have that $\alpha^*=su$, for some $s\in S_C$.
		\end{proof}
		The previous lemma demonstrates that integers $q$ and $\bm{p}$ in good $(1/d)$-normalized approximations to $\balph$ can be recovered by the formulas
		\[q=\Tr(su)\quad\text{and}\quad p_{i}=\Tr(su\alpha_i)\quad\text{for}\quad1\le i\le d,\]
		for certain elements $s\in\Lambda^*$ and units $u\in\mc{Z}_\Lambda^\times$. In sympathy with this fact, for any fixed $s\in\Lambda^*$ we define $q_u=q_u(s)\in\Z$ and $\bm{p}_u=\bm{p}_u(s)\in\Z^d$ by
		\[q_u=\Tr(su)\quad\text{and}\quad p_{u,i}=\Tr(su\alpha_i)\quad\text{for}\quad1\le i\le d.\]
		Now we consider statement (ii) from the beginning of this section. To make this statement precise, we have the following lemma.
		\begin{lemma}\label{lem.NormAppChar2}
			Suppose that $\balph$ and $\Lambda$ are as in the previous lemma, and that $s\in\Lambda^*$. Then, with $q_u$ and $\bm{p}_u$ defined as above, for each $u\in\mc{Z}_\Lambda^\times$, there exist $\gamma_u=\gamma_u(s)\in\R$ and $\bm{\beta}_u=\bm{\beta}_u(s)\in\R^d$ satisfying
			\begin{equation}\label{eqn.MatCurveEqn1}
				|q_u|^{1/d}(q_u\balph-\bm{p}_u)=\gamma_uM(\balph)\bm{\beta}_u,
			\end{equation}
			where $M(\balph)$ is an invertible, real $d\times d$ matrix which only depends on $\balph$, where $\bm{\beta}_u$ lies on the surface in $\R^d$ with equation
			\begin{equation}\label{eqn.SurfEqn}
				\left|x_1\cdots x_r\right|\prod_{i=1}^s\left(x_{r+2i-1}^2+x_{r+2i}^2\right)=1,
			\end{equation}
			and where $\gamma_u$ tends to $|\NN(s)|^{1/d}$, as $u$ tends to infinity along any sequence of dominant units.
		\end{lemma}
		\begin{proof} Assume that $s\not=0$, since otherwise the result is trivial. First of all, from the definitions of $q_u$ and $\bm{p}_u$, we have that
			\begin{align}
				q_u\balph-\bm{p}_u&=\begin{pmatrix}
					\alpha_1 &\alpha_1&\cdots &\alpha_1\\
					\vdots&\vdots&\ddots&\vdots\\
					\alpha_d&\alpha_d&\cdots &\alpha_d
				\end{pmatrix}\sigma(su)-\begin{pmatrix}
				\sigma(\alpha_1)^t\\
				\vdots\\
				\sigma(\alpha_d)^t
			\end{pmatrix}\sigma(su)\nonumber\\
				&=\begin{pmatrix}
					\alpha_1-\sigma_1(\alpha_1) &\cdots &\alpha_1-\sigma_d(\alpha_1)\\
					\vdots&\ddots &\vdots\\
					\alpha_d-\sigma_1(\alpha_d)&\cdots &\alpha_d-\sigma_d(\alpha_d)
				\end{pmatrix}\begin{pmatrix}
				\sigma_1(su)\\
				\vdots\\
				\sigma_d(su)
			\end{pmatrix}.\label{eqn.MatCurveEqn2}
			\end{align}
			Now for $1\le i\le 2s$ and $1\le j\le d$ write 
			\begin{equation}\label{eqn.Defziwij}
				z_i=\sigma_{r+i}(su)\quad\text{and}\quad w_{i,j}=\alpha_j-\sigma_{r+i}(\alpha_j).
			\end{equation}
			Then we have that
			\begin{align*}
				&\begin{pmatrix}
					\alpha_1-\sigma_{r+1}(\alpha_1)&\cdots &\alpha_1-\sigma_{d}(\alpha_1)\\
					\vdots&\ddots&\vdots\\
					\alpha_d-\sigma_{r+1}(\alpha_d)&\cdots &\alpha_d-\sigma_d(\alpha_d)
				\end{pmatrix}\begin{pmatrix}
				\sigma_{r+1}(su)\\
				\vdots\\
				\sigma_d(su)
			\end{pmatrix}\\
				&\hspace*{20bp} =\begin{pmatrix}
					w_{1,1}&\cdots &w_{s,1}&\overline{w_{1,1}}&\cdots&\overline{w_{s,1}}\\
					\vdots&\ddots & \vdots&\vdots&\ddots & \vdots\\
					w_{1,d}&\cdots&w_{s,d}&\overline{w_{1,d}}&\cdots&\overline{w_{s,d}}
				\end{pmatrix}(z_1,\cdots,z_s,\overline{z_1},\cdots,\overline{z_s})^t\\
				&\hspace*{20bp} =2\sum_{i=1}^s\left(\mathrm{Re}(z_iw_{i,1}),\ldots,\mathrm{Re}(z_iw_{i,d})\right)^t\\
				&\hspace*{20bp} =\begin{pmatrix}
					2\mathrm{Re}(w_{1,1})&-2\mathrm{Im}(w_{1,1})&\cdots&2\mathrm{Re}(w_{s,1})&-2\mathrm{Im}(w_{s,1})\\
					\vdots&\vdots&\ddots&\vdots&\vdots\\
					2\mathrm{Re}(w_{1,d})&-2\mathrm{Im}(w_{1,d})&\cdots&2\mathrm{Re}(w_{s,d})&-2\mathrm{Im}(w_{s,d})	
				\end{pmatrix}\begin{pmatrix}
				\mathrm{Re}(z_1)\\ \mathrm{Im}(z_1)\\ \vdots \\\mathrm{Re}(z_s) \\\mathrm{Im}(z_s)
			\end{pmatrix}.
			\end{align*}
			Using this fact in \eqref{eqn.MatCurveEqn2}, and taking
			\begin{equation}\label{eqn.Defgammau}
				\gamma_u=|q_u|^{1/d}\left|\sigma_1(su)\cdots\sigma_d(su)\right|^{1/d},
			\end{equation}
			we have that \eqref{eqn.MatCurveEqn1} holds with 
			\begin{equation}\label{eqn.Defbetau}
				\bm{\beta}_u=\frac{|q_u|^{1/d}}{\gamma_u}\left(\sigma_1(su),\ldots,\sigma_{r}(su),\mathrm{Re}(z_1),\mathrm{Im}(z_1),\ldots,\mathrm{Re}(z_s),\mathrm{Im}(z_s)\right)^t,
			\end{equation}
			and with a real $d\times d$ matrix  $M(\balph)$ which only depends on $\balph$. The vector $\bm{\beta}_u$ is easily seen to satisfy the equation \eqref{eqn.SurfEqn}.
			
			Next, we have that
			\begin{align}\label{eqn.GamEst}
				\gamma_u^d&=\frac{|q_u||\NN(su)|}{|\sigma_0(su)|}=\frac{|\NN(s)|}{|su|}|\Tr(su)|=|\NN(s)|\left|\frac{su}{|su|}+\sum_{i=1}^{d}\frac{\sigma_i(su)}{|su|}\right|.
			\end{align}
			As $u$ tends to infinity along any collection of units satisfying \eqref{eqn.DomUnitDef}, the right hand side of this expression approaches $|\NN(s)|$, so $\gamma_u$ tends to $|\NN(s)|^{1/d}$, as required.

			Finally, to see that $M(\balph)$ is invertible note that, by allowing $s$ to vary in the above argument we can achieve all possible choices of $q\in\Z$ and $\bm{p}\in\Z^d$ as values of $q_u$ and $\bm{p}_u$. For example, we can always take
			\[s=q\alpha_0^*+p_1\alpha_1^*+\cdots +p_d\alpha_d^*\]
			and $u=1$, and we have that $q_u=q$ and $\bm{p}_u=\bm{p}$. However, since $M(\balph)$ depends only on $\balph$, this allows us to deduce from \eqref{eqn.MatCurveEqn1} that its rank must be maximal.
		\end{proof}
		In our proofs of Theorems \ref{thm.QuadratApps} and \ref{thm.CubicApps} in the following two sections, we will make use of the specific values of $\gamma_u$ and $\bm{\beta}_u$ from \eqref{eqn.Defgammau} and \eqref{eqn.Defbetau}, as well as that of
		\begin{equation}\label{eqn.DefMalpha}
			M(\balph)=\begin{pmatrix}
				\alpha_1-\sigma_1(\alpha_1)&\cdots& \alpha_1-\sigma_r(\alpha_1)&2\mathrm{Re}(w_{1,1})&\cdots&-2\mathrm{Im}(w_{s,1})\\
				\vdots&\ddots&\vdots&\vdots&\ddots&\vdots\\
				\alpha_d-\sigma_1(\alpha_d)&\cdots& \alpha_d-\sigma_r(\alpha_d)&2\mathrm{Re}(w_{1,d})&\cdots&-2\mathrm{Im}(w_{s,d})
			\end{pmatrix},
		\end{equation}
		where the $z_i$ and $w_{i,j}$ are given by \eqref{eqn.Defziwij}.
		
		\section{Proof of Theorem \ref{thm.QuadratApps}}\label{sec.QuadThm}
		Let $\alpha=a+b\sqrt{D}\in\R$, with $a,b\in\Q$, $b\not=0$, and $D\ge 2$ a squarefree integer, let $K=\Q(\alpha)$, and let $\Lambda=\Lambda(\alpha_0,\alpha_1)$, with $\alpha_0=1$ and $\alpha_1=\alpha$. First we show that
		\begin{equation}\label{eqn.Thm4PfIncl}
			\mc{A}_1(\alpha)\subseteq\left\{\pm 2b\sqrt{D}\phantom{\cdot}\NN(s) : s\in\Lambda^*\setminus\{0\}\right\}.
		\end{equation}
		Fix any $C>0$ and let $S_C$ be the set from Lemma \ref{lem.NormAppChar1}. The dominant units in $\mc{Z}_\Lambda^\times$ are a uniformly discrete set (the set of all units is a group of rank 1). Therefore we can enumerate the collection of all positive dominant units in $\mc{Z}_\Lambda^\times$ as $\{u_k\}_{k\in\N}$, with $u_k\rar\infty$ as $k\rar\infty$. The negative dominant units are, of course, just the negatives of these. By Lemma \ref{lem.NormAppChar1}, if $q\in\N$ and $p\in\Z$ satisfy
		\[q\left|q\alpha-p\right|\le C\]
		then
		\[q\alpha_0^*+p\alpha_1^*=\pm su_k,\]
		for some $s\in S_C$ and $k\in\N$ and, furthermore, we have that
		\[q=q_{u_k}=\Tr(\pm su_k)\quad\text{and}\quad p=p_{u_k}=\Tr(\pm su_k\alpha).\]
		By Lemma \ref{lem.NormAppChar2}, together with the definition of $M(\alpha)$ from \eqref{eqn.DefMalpha}, we have for each $s\in S_C$ and $k\in\N$ that
		\begin{equation}\label{eqn.Thm4PfNormAppForm}
		q_{u_k}(q_{u_k}\alpha-p_{u_k})=\pm2b\sqrt{D}\gamma_{u_k},
		\end{equation}
		and that $\gamma_{u_k}\rar |\NN(s)|$ as $k\rar\infty$. This establishes the inclusion in \eqref{eqn.Thm4PfIncl}. The reader may also wish to note that $0\notin\mc{A}_1(\alpha)$, by virtue of the fact that $\alpha$ is badly approximable.
		
		To prove the reverse inclusion, it is enough to note that, for any $s\in\Lambda^*\setminus\{0\}$ and for any $k\in\N$, the number $\gamma_{u_k}$ above is never equal to $|\NN(s)|$. This follows from formula \eqref{eqn.GamEst}, and it guarantees that at least one of the points $\pm2b\sqrt{D}|\NN(s)|$ is an accumulation point of the sequence in \eqref{eqn.Thm4PfNormAppForm}. Since, by definition, the set $\mc{A}_1(\alpha)$ is symmetric about 0, this completes the proof.

		\section{Proof of Theorem \ref{thm.CubicApps}}\label{sec.CubicThm}
		Suppose that $\bm{\alpha}\in\R^2$ and that $\{1,\alpha_1,\alpha_2\}$ is a basis for a cubic field $K/\Q$, and let $\Lambda=\Lambda(\alpha_0,\alpha_1,\alpha_2)$, with $\alpha_0=1$. We divide the proof of Theorem \ref{thm.CubicApps} into two cases, depending on whether or not $K$ has a non-trivial complex embedding.
		
		\noindent{\textbf{Complex case:}} Suppose that $K$ has a non-trivial embedding into $\C$. Using Lemmas \ref{lem.NormAppChar1} and \ref{lem.NormAppChar2}, and essentially the same argument given in the previous section, we find that $\mc{A}_{1/2}(\balph)$ is a subset of
		\begin{equation}\label{eqn.Thm5AccPts1}
			\bigcup_{s\in \Lambda^*\setminus\{0\}}\left|\NN (s)\right|^{1/2}\mc{C}_K,
		\end{equation}
		where $\mc{C}_K$ is the ellipse defined by
		\[\mc{C}_K=\left\{M(\balph)\bm{x}:\bm{x}\in\R^2, x_1^2+x_2^2=1\right\}.\]
		
		To prove the reverse inclusion, let $u>1$ be a fundamental unit for $\mc{Z}_\Lambda^\times$. Then
		\[\sigma_1(u)=u^{-1/2}e(\theta),\]
		for some $\theta\in\R$, where $e(z)=e^{2\pi iz}$. Observe that $\theta$ must be irrational, since otherwise we would have that $\sigma_1(u)^k=\sigma_2(u)^k\in\R$ for some $k\in\N$. This in turn would imply that $u^k\in\Q$, which would mean that $u^k=\pm 1$, a contradiction.
		
		Now suppose that $s\in\Lambda^*\setminus\{0\}$ and write
		\[\sigma_1(s)=\rho_s e(\psi_s),\]
		with $\rho_s,\psi_s\in\R$. For each $k\in\N$ and $i=1,2$ let
		\[q_k=\Tr(su^k)\quad\text{and}\quad p_{k,i}=\Tr(su^k\alpha_i).\]
		By Lemma \ref{lem.NormAppChar2}, we have that
		\[|q_k|^{1/2}(q_k\balph-\bm{p}_k)=\gamma_kM(\balph)\bm{\beta}_k,\]
		where $\gamma_k\rar |\NN(s)|^{1/2}$ as $k\rar\infty$ and, by equations \eqref{eqn.Defgammau} and \eqref{eqn.Defbetau},
		\[\bm{\beta}_k=\frac{1}{|\sigma_1(su^k)|}\left(\mathrm{Re}(su^k),\mathrm{Im}(su^k)\right)^t=\left(\cos (\psi_s+k\theta),\sin (\psi_s+k\theta)\right)^t.\]
		Since $\theta$ is irrational, the sequence $\{\bm{\beta}_k\}_{k\in\N}$ is dense on the unit circle in $\R^2$. Also, as before, by equation \eqref{eqn.GamEst} the numbers $\gamma_k$ are never actually equal to $|\NN(s)|^{1/2}$. Therefore every point in \eqref{eqn.Thm5AccPts1} is an accumulation point of $\mc{A}_{1/2}(\balph).$

		\noindent{\textbf{Totally real case:}} If $K$ does not have a non-trivial embedding into $\C$ then, by the same arguments as before, we find that $\mc{A}_{1/2}(\balph)$ is a subset of
		\begin{equation*}
		\bigcup_{s\in \Lambda^*\setminus\{0\}}\left|\NN (s)\right|^{1/2}\left(\mc{C}^+_K\cup\mc{C}^-_K\right),
		\end{equation*}	
		where $\mc{C}_K^+$ and $\mc{C}_K^-$ are the conjugate hyperbolas defined by
		\[\mc{C}^\pm_K=\left\{M(\balph)\bm{x}:\bm{x}\in\R^2, x_1x_2=\pm 1\right\}.\]
		
		Let $u_1$ and $u_2$ be a pair of fundamental units for $\mc{Z}_\Lambda^\times$ and, for each $\bm{a}\in\Z^2$, write
		\[u_{\bm{a}}=u_1^{2a_1}u_2^{2a_2}.\]
		First we establish the following result, which is a key ingredient in the proof.
		\begin{lemma}\label{lem.TotRealUnitDist}
		For any constants $0<c_1<c_2$ the set
		\begin{equation}\label{eqn.UnitHypSection}
			\mc{U}(c_1,c_2)=\left\{u_{\bm{a}}>1:\bm{a}\in\Z^2, c_1<\frac{\sigma_1(u_{\bm{a}})}{\sigma_2(u_{\bm{a}})}<c_2\right\}
		\end{equation}
		is infinite and uniformly discrete in $(0,\infty)$.
		\end{lemma}
		\begin{proof}
		First of all, to see that this set is uniformly discrete, consider the logarithmic embedding $\varphi:\mc{Z}_\Lambda^\times\rar\R^3$. For any $\bm{a}\in\Z^2$, the point $u_{\bm{a}}$ will be in $\mc{U}(c_1,c_2)$ if and only if
		\begin{equation*}
			\varphi(u_{\bm{a}})\in \mc{R}_>(c_1,c_2),
		\end{equation*}
		where
		\[\mc{R}_>(c_1,c_2)=\{(x_0,x_1,x_2)\in\R^3:x_0>0, ~\log c_1<x_1-x_2<\log c_2\}.\]
		For any $c_3>0$, the intersection of the hyperplane
		\begin{equation}\label{eqn.TotRealHypEqn}
			x_0+x_1+x_2=0
		\end{equation}
		with the region
		\[\{(x_0,x_1,x_2)\in\R^3:0<x_0<c_3, ~\log c_1<x_1-x_2<\log c_2\}\]
		is bounded, therefore it contains finitely many points of $\varphi(\mc{Z}_\Lambda^\times)$. Since $\varphi(\mc{Z}_\Lambda^\times)$ is a lattice in this hyperplane, this argument easily implies that the $x_0$ coordinates of the lattice points in $\mc{R}_>(c_1,c_2)$ are a uniformly discrete set. This is in turn implies that $\mc{U}(c_1,c_2)$ is uniformly discrete (in fact, its elements are spaced apart at least geometrically).
		
		To prove that $\mc{U}(c_1,c_2)$ is infinite, it is enough to show that, for any $0<c_1<c_2$, it is non-empty. Suppose to the contrary that $0<c_1<c_2$ and that
		\begin{equation}\label{eqn.TotRealUnitsDist2}
			\varphi\left(\left\{u_{\bm{a}}:\bm{a}\in\Z^2\right\}\right)\cap\mc{R}_>(c_1,c_2)=\emptyset.
		\end{equation}
		Under this assumption it is not difficult to show that we also must have
		\begin{equation}\label{eqn.TotRealUnitsDist3}
		\varphi\left(\left\{u_{\bm{a}}:\bm{a}\in\Z^2\setminus{\bm{0}}\right\}\right)\cap\mc{R}_\le(c_1,c_2)=\emptyset,
	\end{equation}
		where
		\[\mc{R}_\le(c_1,c_2)=\{(x_0,x_1,x_2)\in\R^3:x_0\le 0,~\log c_1<x_1-x_2<\log c_2\}.\]
		To see this, note that $\varphi\left(\left\{u_{\bm{a}}:\bm{a}\in\Z^2\right\}\right)$ is a lattice in the hyperplane \eqref{eqn.TotRealHypEqn}, and suppose there were a non-zero point $\bm{\lambda}=(\lambda_0,\lambda_1,\lambda_2)$ of this lattice in $\mc{R}_\le(c_1,c_2)$. By the Minkowski convex body theorem \cite[Appendix B, Theorem 2]{Cass1957}, for any $\epsilon>0$ there is a non-zero point $\bm{\lambda}'=(\lambda_0',\lambda_1',\lambda_2')$ of the lattice with
		\[|\lambda_1'-\lambda_2'|<\epsilon.\]
		By choosing $\epsilon$ sufficiently small, we can ensure that $\lambda_0'>|\lambda_0|$ and that
		\[\log c_1<(\lambda_1-\lambda_2)+(\lambda_1'-\lambda_2')<\log c_2.\]
		This implies that $\bm{\lambda}+\bm{\lambda}'\in\mc{R}_>(c_1,c_2)$, which contradicts  \eqref{eqn.TotRealUnitsDist2}. Therefore if \eqref{eqn.TotRealUnitsDist2} holds then there is no non-zero point of $\varphi\left(\left\{u_{\bm{a}}:\bm{a}\in\Z^2\right\}\right)$ in the set
		\begin{equation}
			\mc{R}(c_1,c_2)=\{(x_0,x_1,x_2)\in\R^3:\log c_1<x_1-x_2<\log c_2\}.
		\end{equation}
		
		Still assuming that \eqref{eqn.TotRealUnitsDist2} holds, note that the intersection of $\mc{R}(c_1,c_2)$ with the hyperplane \eqref{eqn.TotRealHypEqn} has infinite volume. Therefore, by Blichfeldt's theorem \cite[Appendix B, Theorem 1]{Cass1957}, there are distinct points $\bm{y}$ and $\bm{y}'$ in this intersection with
		\[\bm{y}-\bm{y}'=\bm{\xi}=(\xi_0,\xi_1,\xi_2)\in\varphi\left(\left\{u_{\bm{a}}:\bm{a}\in\Z^2\right\}\right).\]
		By replacing $\bm{\xi}$ with its negative, if necessary, we then have that
		\[0\le \xi_1-\xi_2<\log c_2-\log c_1.\]
		If $\xi_1-\xi_2>0$ then an integer multiple of $\bm{\xi}$ lies in $\mc{R}(c_1,c_2),$ which gives a contradiction. Therefore we must have that $\xi_1=\xi_2$. This implies that there is a non-zero $\bm{a}\in\Z^2$ with
		\[|\sigma_1(u_{\bm{a}})|=|\sigma_2(u_{\bm{a}})|.\]
		However this cannot happen because:
		\begin{itemize}[itemsep=5bp,parsep=5bp, topsep=0bp]
			\item[(i)] The unit $u_{\bm{a}}$ is an irrational number in a cubic extension of $\Q$. Therefore it is also a cubic irrational, which by separability implies that $\sigma_1(u_{\bm{a}})\not=\sigma_2(u_{\bm{a}})$.
			\item[(ii)] If $\sigma_1(u_{\bm{a}})=-\sigma_2(u_{\bm{a}})$ then
			\[\Tr(u_{\bm{a}})=u_{\bm{a}}\notin\Q,\]
			which is a contradiction.
		\end{itemize}
		We conclude that, for $0<c_1<c_2$, equation \eqref{eqn.TotRealUnitsDist2} cannot hold. This therefore completes the proof of the lemma.
		\end{proof}
		Next suppose that $s\in\Lambda^*\setminus\{0\}$ and, for each $\bm{a}\in\Z^2$ and $i=1,2$ write
		\[q_{\bm{a}}=\Tr\left(su_{\bm{a}}\right)\quad\text{and}\quad p_{\bm{a},i}=\Tr\left(su_{\bm{a}}\alpha_i\right).\]
		The totally real case of Theorem \ref{thm.CubicApps} will follow almost immediately from the following result.
		\begin{lemma}\label{lem.TotRealAccumPtStruc}
		For any $s\in\Lambda^*\setminus\{0\}$ and $C>0$, the collection of accumulation points of
		\begin{equation}\label{eqn.HypAccumPts}
			\{\pm|q_{\bm{a}}|^{1/2}(q_{\bm{a}}\balph-\bm{p}_{\bm{a}}):\bm{a}\in\Z^2, u_{\bm{a}}>1\}\cap [-C,C]^2
		\end{equation}
		is either
		\begin{equation*}
			\left|\NN (s)\right|^{1/2}\mc{C}_K^+\cap[-C,C]^2\quad\text{or}\quad\left|\NN (s)\right|^{1/2}\mc{C}_K^-\cap[-C,C]^2,
		\end{equation*}
		depending on whether $\sigma_1(s)\sigma_2(s)$ is positive or negative.
		\end{lemma}
		\begin{proof}					
		Suppose for simplicity that $\sigma_1(s)$ and $\sigma_2(s)$ are both positive (the other cases are all of equal difficulty to handle). By Lemma \ref{lem.NormAppChar2}, for each $\bm{a}\in\Z^2$ we have that
		\[|q_{\bm{a}}|^{1/2}(q_{\bm{a}}\balph-\bm{p}_{\bm{a}})=\gamma_{\bm{a}}M(\balph)\bm{\beta}_{\bm{a}},\]
		where 
		\[\gamma_{\bm{a}}=|q_{\bm{a}}|^{1/2}|x_{\bm{a},1}x_{\bm{a},2}|^{1/2}\]
		and
		\[\bm{\beta}_{\bm{a}}=\left(\left(\frac{x_{\bm{a},1}}{x_{\bm{a},2}}\right)^{1/2},\left(\frac{x_{\bm{a},2}}{x_{\bm{a},1}}\right)^{1/2}\right)^t,\]
		with
		\[x_{\bm{a},j}=\sigma_j(su_{\bm{a}})\quad\text{for}\quad j=1,2.\]
		Note that by assumption here, $x_{\bm{a},j}>0$ for $j=1,2$. The set of units $u_{\bm{a}}>1$, with $\bm{a}\in\Z^2,$ for which
		\[\pm|q_{\bm{a}}|^{1/2}(q_{\bm{a}}\balph-\bm{p}_{\bm{a}})\in [-C,C]^2\]
		is a subset of $\mc{U}(c_1,c_2)$, for some $0<c_1<c_2$, where $\mc{U}(c_1,c_2)$ is defined by \eqref{eqn.UnitHypSection}. Therefore by Lemma \ref{lem.TotRealUnitDist} this set of units is a uniformly discrete subset of $(1,\infty)$. It follows from formula \eqref{eqn.GamEst} that any accumulation point of \eqref{eqn.HypAccumPts} must lie on $\left|\NN (s)\right|^{1/2}\mc{C}_K^+\cap[-C,C]^2$. It is also clear from the same formula that none of the points of \eqref{eqn.HypAccumPts} actually lie on this section of the hyperbola.
		
		All that remains is to show that, for all $\epsilon>0$, any $\epsilon$-neighborhood of an arc of the hyperbola above contains a point of \eqref{eqn.HypAccumPts}. Given any such arc, we may first choose $0<c_1<c_2$ so that for every sufficiently large $u_{\bm{a}}\in\mc{U}(c_1,c_2)$, the points
		\[\pm|q_{\bm{a}}|^{1/2}(q_{\bm{a}}\balph-\bm{p}_{\bm{a}})\]
		all lie in this arc (with the plus or minus sign chosen appropriately). Then (since by Lemma \ref{lem.TotRealUnitDist}, the set $\mc{U}(c_1,c_2)$ is infinite and uniformly discrete), we can let $u_{\bm{a}}$ tend to infinity to ensure that the corresponding normalized approximations all lie within the $\epsilon$-neighborhood of this arc. This completes the proof.
		\end{proof}
		
		Finally, for each $s\in\Lambda^*\setminus\{0\}$ we define $\ell_s:\R^2\rar\R^2$ to be the identity map if $\sigma_1(s)\sigma_2(s)>0$, and an invertible linear map from $\mc{C}_K^+$ to $\mc{C}_K^-$ otherwise. Lemma \ref{lem.TotRealAccumPtStruc} guarantees that
		\begin{equation}\label{eqn.Thm5AccPts2}
		\bigcup_{s\in \Lambda^*\setminus\{0\}}\left|\NN (s)\right|^{1/2}\ell_s(\mc{C}^+_K)\subseteq\mc{A}_{1/2}(\balph).
		\end{equation}	
		Using the notation of Lemmas \ref{lem.NormAppChar1} and \ref{lem.NormAppChar2}, for each $s$, it is possible that there could be other accumulation points of the set
		\[\left\{\pm|q_u|^{1/2}(q_{u}\balph-\bm{p}_{u}):u\in\mc{Z}_\Lambda^\times, u \text{ a dominant unit}\right\},\]
		which do not lie on $|N(s)|^{1/2}\ell_s(\mc{C}^+_K)$. If this happens then, by the arguments above, there would have to be a unit $\tilde{u}$ with $\sigma_1(\tilde{u})\sigma_2(\tilde{u})<0$, and the extra accumulation points would lie on the complementary hyperbola. But then the element $\tilde{s}=s\tilde{u}\in\Lambda^*\setminus\{0\}$ would satisfy
		\[\mathrm{sgn}(\sigma_1(s)\sigma_2(s))=-\mathrm{sgn}(\sigma_1(\tilde{s})\sigma_2(\tilde{s})),\]
		which means that the extra accumulation points in this case would already have been accounted for on the left hand side of \eqref{eqn.Thm5AccPts2} by the $\tilde{s}$ term in the union. This therefore completes the proof of Theorem \ref{thm.CubicApps}.

		\vspace{.15in}
		
		{\footnotesize
			\noindent
			Department of Mathematics\\
			University of Houston\\
			Houston, TX, United States\\
			haynes@math.uh.edu, kkavita@cougarnet.uh.edu
			
		}

	\end{document}